\documentclass{amsart}
\pdfoutput=1
\usepackage{amssymb}
\usepackage[dvipsnames]{xcolor}
\usepackage{tikz-cd}
\usepackage{tikz,pgffor}
\usepackage{ifthen}
\usetikzlibrary{calc}
\usepackage[unicode,colorlinks=true,citecolor=Emerald,linkcolor=Fuchsia]{hyperref}
\usepackage{mathtools} 
\usepackage{enumitem}
\usepackage{mathrsfs}
\usepackage{microtype}

\DeclareSymbolFont{altcal}{U}{eus}{m}{n}
\DeclareMathSymbol{\altc}{\mathord}{altcal}{`C}

\newtheorem{theorem}{Theorem}[section]
\newtheorem{proposition}[theorem]{Proposition}
\newtheorem{lemma}[theorem]{Lemma}
\newtheorem{conjecture}[theorem]{Conjecture}
\theoremstyle{definition}
\newtheorem{example}[theorem]{Example}
\newtheorem{definition}[theorem]{Definition}

\theoremstyle{remark}
\newtheorem{remark}[theorem]{Remark}

\numberwithin{equation}{section}

\DeclareMathOperator{\ndomorphisms}{End}
\DeclareMathOperator{\colors}{Col}
\DeclareMathOperator{\aut}{Aut}

\DeclareMathOperator{\id}{id}
\newcommand{\uc}{\underline{c}}
\newcommand{\ud}{\underline{d}}
\newcommand{\cyc}{\mathbf{Cyc}}
\newcommand{\cycne}{\mathbf{Cyc}^{\uparrow}}
\newcommand{\cat}{\mathbf{Cat}}
\newcommand{\operads}{\mathbf{Opd}}
\newcommand{\sset}{\mathbf{sSet}}
\newcommand{\set}{\mathbf{Set}}
\newcommand{\iset}{\mathbf{iSet}}
\newcommand{\comp}[2]{\circ_{#1}^{#2}}
\newcommand{\numberfunctor}{\mathscr{N}}

\newcommand{\slh}[3]{ 
	#1 - #1 * #3 + #2 * #3
}

\newcommand{\curvemod}[7]{
	\ifthenelse
	{#7=0}
	{\curve{#1}{#2}{#3}{#4}{#5}{#6}}
	{\draw[#6] (#2: #3)--(#4:#5);}
}
\newcommand{\curve}[6]{
	\xdef\mycoords{}
	\foreach \step in
	{0,...,#1}
	{\xdef\mycoords{\mycoords 
		(\slh{#2}{#4}{\step/#1} :\slh{#3}{#5}{\step/#1})
	};
	}
	\draw[#6] plot [smooth] coordinates {\mycoords};
}

\title{Dwyer--Kan homotopy theory for cyclic operads}
\author[G.~C.~Drummond-Cole]{Gabriel C. Drummond-Cole}
\thanks{The first author was supported by IBS-R003-D1.}
\email{gabriel.c.drummond.cole@gmail.com}
\author[P.~Hackney]{Philip Hackney}
\thanks{The second author acknowledges the support of Australian Research Council Discovery Project grant DP160101519.}
\address{Department of Mathematics\\ University of Louisiana at Lafayette\\ Lafayette, LA, USA}
\email{philip@phck.net} 
\urladdr{http://phck.net}

\begin{document}

\begin{abstract}
We introduce a general definition for colored cyclic operads over a symmetric monoidal ground category, which has several appealing features.
The forgetful functor from colored cyclic operads to colored operads has both adjoints, each of which is relatively simple.
Explicit formulae for these adjoints allow us to lift the Cisinski--Moerdijk model structure on the category of colored operads enriched in simplicial sets to the category of colored cyclic operads enriched in simplicial sets.
\end{abstract}

\maketitle

\section{Introduction}
This paper has three main goals:
\begin{enumerate}
\item to establish a general definition for colored cyclic operads,
\item to prove the existence of a Dwyer--Kan type model structure on the category of positive cyclic operads in $\sset$, and
\item to advertise and publicize the existence of an exceptional right adjoint to the forgetful functor from cyclic operads to operads.
\end{enumerate}
We will briefly discuss and contextualize each of these goals and then outline the structure of the remainder of the paper.
\subsection*{Colored cyclic operads}
Colored operads, or symmetric multicategories, are more or less familiar to mathematicians working in operads and in some parts of category theory. 
Cyclic operads are also well-studied by the former group.
However, the notion of colored cyclic operad has appeared less frequently in the literature.
A partial catalog appears in Section~\ref{section: relation to other definitions}, where we compare our definition to other extant definitions.
In our viewpoint, a general definition of colored cyclic operads should have all of the following features:
\begin{enumerate}
\item it should be categorical, i.e., work over all or most symmetric monoidal ground categories,\label{point categorical}
\item it should be symmetric, i.e., allow for an action of the full symmetric group and not only the cyclic group on its operation objects,
\item it should ``allow the empty profile'' corresponding to operations governed by trees with no leaves, and
\item it should be (anti-)involutive, i.e., color matching for the composition operations should be governed by an involution. 
\end{enumerate}
Shulman, in a paper roughly contemporaneous with this one, arrived at a definition equivalent to ours. See \cite[Example 7.7]{Shulman:2CDCPMA} for a comparison of terminology.
Prior to that, the definition in the literature that came closest to achieving these features was that of a `compact symmetric multicategory' from \cite{JoyalKock:FGNTCSM}, which contains our axioms (in $\set$) as a substructure.
The final item in this list means that our color sets are equipped with an involution, which may not match operadic readers' intuition.
This is a familiar idea from the categorical perspective, however.
For instance, the `cyclic multicategories' of \cite{ChengGurskiRiehl:CMMAM} are essentially the non-symmetric strictly positive version of our cyclic operads (see Proposition~\ref{prop: CGR comparison}).

The extra flexibility of this involution allows us to consider many situations on the same footing, including involution-free definitions of colored cyclic operads as well as colored operads and colored dioperads.
This involution also means that every colored cyclic operad has an underlying anti-involutive category, rather than an underlying dagger category.
Thus, in the present paper, we are trying to subsume the homotopy theory of anti-involutive categories and not that of dagger categories, which should be substantially different (see \cite[2.11]{DrummondColeHackney:CFERIQMSAAIIC} for an explanation).

\subsection*{Dwyer--Kan structures}
Dwyer--Kan equivalences were introduced (not under that name) in \cite{DwyerKan:FCHA} as a notion of weak equivalence between simplicially-enriched categories. Bergner \cite{Bergner:MCSCSC} showed that the collection of all simplicially-enriched categories can be endowed with a Quillen model structure whose weak equivalences are the Dwyer--Kan equivalences.
This constitutes a model for $(\infty,1)$-categories (see the survey \cite{Bergner:SI1C} for an overview of commonly used models and equivalences among them).

The Bergner model structure on simplicial categories has been extended in subsequent years to a variety of other settings, each of which can be equipped with a variation of the notion of Dwyer--Kan equivalence, including operads~\cite{CisinskiMoerdijk:DSSO,Robertson:HTSEM,Caviglia:MSECO}, props~\cite{HackneyRobertson:HTSP,Caviglia:DKMSECP}, dioperads and properads~\cite{HackneyRobertsonYau:SMIP}, wheeled operads, wheeled properads \cite{Yau:DKHTAOC}, and others.
All of these settings have something to do with directed graphs.
Missing from this picture were \emph{undirected} settings, such as anti-involutive categories%
\footnote{We note that in 2014 Giovanni Caviglia proposed a model structure for $\mathcal{V}$-enriched \emph{dagger categories} \cite{Caviglia:DC}, which, unfortunately, never appeared in final form.}, cyclic operads, and modular operads.
The first of these was addressed in our prior paper \cite{DrummondColeHackney:CFERIQMSAAIIC}, and in the present paper we utilize similar techniques to establish a Dwyer--Kan type model structure for positive cyclic operads (Theorem~\ref{dk model structure theorem}).

\subsection*{The exceptional right adjoint}
Our proof uses the existence of an unusual right adjoint to the forgetful functor from colored cyclic operads to colored operads. 
The existence of this kind of adjoint was first established in the monochrome case (in $\set$) in Templeton's unpublished thesis~\cite{Templeton:SGO}.
This construction was apparently forgotten until recently (see Remark~\ref{remark: right and left adjoints}).
The existence of this right adjoint allows us to use a criterion of~\cite{DrummondColeHackney:CFERIQMSAAIIC} to quickly prove the main theorem. 
We consider this short proof an application of the existence of this exceptional right adjoint, which we would like to advertise.
In \cite{DrummondColeHackney:CSOT}, we characterize those maps of colored operads so that pullback of algebras admits a right adjoint.
This recovers Templeton's result, but does not suffice to recover the adjoint from the present paper, which is a kind of ``globalized'' version with variable color set.

\subsection*{Organization}
We now give a brief overview of the contents of this paper.
Section~\ref{section cyclic operads} contains our definition of cyclic operads, as well as several examples (which may be safely skipped).
The purpose of Section~\ref{section adjunctions} is to describe both adjoints to the forgetful functor $\cyc \to \operads$.
This section is fundamental to our approach to the proof of the main theorem.
In Section~\ref{section: relation to other definitions}, we explain how our definition is related to various others in the literature.
The main results of the paper do not depend on these relationships, so the reader can freely skip them if she is not already familiar with other definitions of cyclic operads.
That said, Definition~\ref{def cycne} introduces the notion of positive cyclic operads, a major topic of the rest of the paper.
Thus the first page or so of Section~\ref{section: relation to other definitions} should not be skipped.
Section~\ref{section bifibrations} exists primarily to contain Proposition~\ref{cyc is bicomplete}; there are other ways to prove this proposition, but we have chosen to use Grothendieck bifibrations as our tool.
The final section contains the main theorem, its proof, and a few comments about interesting related topics.
\subsection*{Acknowledgments}
The authors would like to thank Alexander Campbell, Richard Garner, Joachim Kock, Jovana Obradovi\'c, Marcy Robertson, and Ben Ward for useful discussions.

\section{Colored operads and cyclic operads}
\label{section cyclic operads}

\subsection{Notation and Conventions}
We will use interval notation $[m,n]$ to denote intervals \emph{of integers}.
Thus we can write the symmetric group $\Sigma_n$ as $\aut([1,n])$, the group of bijections of the set $\{1, 2, \dots, n \}$.
We consider $\Sigma_n$ as a subset of $\Sigma_n^+ \coloneqq \aut([0,n])$ in the natural way, namely as the set of bijections which fix the element $0$.
Denote by $\tau_j \in \Sigma_{j-1}^+$ the automorphism of $[0,j-1]$ defined by $\tau_j(t) = t+1 \mod j$. 
We will often abbreviate $\tau_j$ to $\tau$ to avoid a proliferation of indices.

Let us fix some conventions about lists of colors.
If $\mathfrak{C}$ is a set (called the set of \emph{colors}), then a \emph{profile} $\uc = c_1, \ldots, c_n$ is just an ordered list of elements of $\mathfrak{C}$; write $|\uc| = n$ for the length of the ordered list.
If $\sigma \in \Sigma_n$, then define $\uc \sigma = c_{\sigma(1)}, \dots, c_{\sigma(n)}$ for the permuted list.
Sometimes we want profiles to be indexed from $0$ rather than $1$; if $\uc = c_0, c_1, \ldots, c_n$ we write $\|\uc\| = n.$ If $\sigma \in \Sigma_n^+$ we write $\uc \sigma = c_{\sigma(0)}, c_{\sigma(1)}, \dots, c_{\sigma(n)}$.

Throughout, $\mathcal{V}$ will denote a closed symmetric monoidal category.
For an arbitrary such monoidal category, we will write $\otimes$ for the tensor, $\mathbf{1}_{\mathcal{V}}$ for the tensor unit, and $\multimap$ for the internal hom.
Monadic definitions of (cyclic) operads require at least some colimits to exist in $\mathcal{V}$; this is not necessary for the definitions in this section.

\subsection{Colored operads} 
As above, suppose that $\mathfrak{C}$ is a set.
Given two profiles $\uc = c_1, \dots, c_n$ and $\ud = d_1, \dots, d_m$ in $\mathfrak{C}$, define (for $1\leq i \leq n$)
\[
	(c_1, \ldots, c_n) \comp{i}{} (d_1, \dots, d_m) = c_1, \dots, c_{i-1}, d_1, \dots, d_m, c_{i+1}, \dots, c_n.
\]
A $\mathfrak{C}$-colored operad $P$ in $\mathcal{V}$ (a closed symmetric monoidal category) consists of
\begin{itemize}
	\item a set of objects $P(\uc; c) = P(c_1,\dots, c_n; c) \in \mathcal{V}$ (where $c_i,c\in \mathfrak{C}$)
	\item for each $\sigma\in\Sigma_n$ a map $\sigma^* : P(c_1, \dots, c_n; c) \to P(c_{\sigma(1)}, \dots, c_{\sigma(n)}; c)$ satisfying $(\sigma')^*\sigma^* = (\sigma\sigma')^*$ and $(\id)^*=\id$,
	\item identity elements $\id_c : \mathbf{1}_{\mathcal{V}} \to P(c;c)$, and 
	\item composition operations (defined when $c_i = d$) \[
		\comp{i}{} : P(\uc; c) \otimes P(\ud; d) \to P(\uc \comp{i}{} \ud; c).
	\]
\end{itemize}
Let us abuse notation and write the axioms in terms of elements (see Remark~\ref{remark on elements}):
\begin{itemize}
	\item Associativity is satisfied:
	\[
(x\circ_j y) \circ_i z = \begin{cases}
(x\circ_i z) \circ_{j+|z|-1} y & 1\leq i \leq j-1 \\
x \circ_j (y\circ_{i-j+1} z) &  j\leq i \leq  j+|y| - 1\\
(x\circ_{i-|y|+1} z) \circ_j y & j+ |y| \leq i \leq |x| + |y| - 1
\end{cases}
\]
\item The identity elements are identities, that is, $\id_c \comp{1}{} x = x$ and $x \comp{i}{} \id_c = x$ whenever these compositions are defined
\item $(\sigma_1^*x) \circ_i (\sigma_2^* y) = \sigma^* (x \circ_{\sigma(i)} y)$ whenever $\sigma_1, \sigma_2, \sigma$ are elements of symmetric groups which fit into a diagram
\[ \begin{tikzcd}
{[1,n]\setminus \{i\}} \amalg {[1,m]} 
\rar{\sigma_1 \amalg \sigma_2} 
\dar 
& 
{[1,n]\setminus \{\sigma_1(i)\}} \amalg {[1,m]} 
\dar
\\
{[1,i-1]} \amalg {[1,m]} \amalg {[i+1,n]}
\dar
&
{[1,\sigma_1(i)-1]} \amalg {[1,m]} \amalg {[\sigma_1(i)+1,n]}
\dar
\\
{[1,n+m-1]} 
\rar[swap]{\sigma} 
&
{[1,n+m-1]}
\end{tikzcd}\]
\end{itemize}
of bijections,
where the upper vertical maps rearrange and the lower vertical maps are `order preserving'.

Given a $\mathfrak{C}$-colored operad $P$, we will write $\colors(P)$ for the set $\mathfrak{C}$.
Given two colored operads $P$ and $Q$ (with unspecified sets of colors), a \emph{morphism} from $P$ to $Q$ consists of a function $f : \colors(P) \to \colors(Q)$ and maps 
\[
	P(\uc; c) = P(c_1,\dots, c_n; c) \to Q(fc_1, \dots, fc_n; fc) = Q(f\uc; fc)
\]
in $\mathcal{V}$ which are compatible with all of the structure.
We will write $\operads(\mathcal{V})$ for the category of all colored operads, omitting the parameter when $\mathcal{V}$ is understood from context. 
If $\mathfrak{C}$ is a set, we write $\operads_{\mathfrak{C}}(\mathcal{V})$ for the subcategory whose objects are those operads $P$ with $\colors(P) = \mathfrak{C}$ and maps those which are the identity on colors.
\subsection{Colored cyclic operads}
\label{subsection: colored cyclic operads}
In order to cut down on the number of cases involved in giving a definition of cyclic operads, we will rely on certain helper functions.
Define functions $\alpha_{i,m}^n$ (abbreviated $\alpha_{i,m}$) and $\beta_{j,i}^{m,n}$ (abbreviated $\beta_{j,i}$), for $0\leq i \leq n$ and $0 \leq j \leq m$, as
\begin{align*}
	\alpha_{i,m} = \alpha_{i,m}^n : [0,n] \setminus \{i\} &\to [0,n+m-1] \\
	\beta_{j,i} = \beta_{j,i}^{m,n}: [0,m] \setminus \{j\} &\to [0,n+m-1]
\end{align*}
by
\begin{align*}
	\alpha_{i,m}|_{[0,i-1]} &= (x \mapsto x) & \alpha_{i,m}|_{[i+1,n]} &= (x \mapsto x-1+m) \\
	\beta_{j,i}|_{[0,j-1]} &= (x \mapsto x+m-j+i) & \beta_{j,i}|_{[j+1,m]} &= (x \mapsto x-j-1+i). \\
\end{align*}
The function $\alpha_{i,m} \amalg \beta_{j,i}$ is a bijection.
Notice that $\beta_{0,i} : [1,m] \to [0,n+m-1]$ is given by $x \mapsto x+i-1$ and $\beta_{j,i} = \beta_{0,i} \tau_{m+1}^{-j}$.

Given two profiles, write
\begin{multline*}
(c_0, \ldots, c_n) \comp{i}{j} (d_0, \dots, d_m) = e_0, \dots, e_{m+n-1}
\\ = c_0, \dots, c_{i-1}, d_{j+1}, \dots, d_m, d_0, \dots, d_{j-1}, c_{i+1}, \dots, c_n
\end{multline*}
where 
\begin{equation}\label{equation ep}
	e_p = \begin{cases}
		c_{\alpha_{i,m}^{-1}(p)} & p \notin [i,m+i-1] \\
		d_{\beta_{j,i}^{-1}(p)} & p \in [i,m+i-1]
	\end{cases}
	= \begin{cases}
		c_{\alpha_{i,m}^{-1}(p)} & p \in \tau_{n+m}^{m+i}[0,n-1] \\
		d_{\beta_{j,i}^{-1}(p)} & p \in \tau_{n+m}^i[0,m-1]
	\end{cases}
\end{equation}
When giving a linear order to the natural cyclic order of the list of colors, this formula  prioritizes $\uc$ over $\ud$; the term
\[
	(\uc \comp{i}{j} \ud) \tau^{m-j+i} = d_0, \dots, d_{j-1}, c_{i+1}, \dots, c_n, c_0, \dots, c_{i-1}, d_{j+1}, \dots, d_m
\]
is another natural choice of total order. See, e.g.,~Axiom~\eqref{cyclic swap axiom} of Definition~\ref{definition: cyclic operad} for an asymmetry arising from this choice. Morally, this issue appears because we used linearly ordered (rather than cyclically ordered) lists.

We have the following equalities:
\begin{itemize}
	\item $\uc \comp{i}{0} (d_0, \ud) = \uc \circ_i \ud$.
	\item $\uc \comp{i}{1} (c_i, d) = \uc$ 	\item 
	$(d_j, c) \comp{1}{j} \ud = d_j, d_{j+1}, \dots, d_m, d_0, \dots, d_{j-1} = \ud \tau^j$.
							\end{itemize}

\begin{definition}
\label{definition: ccc}
	A (colored) \emph{collection} in $\mathcal{V}$ consists of the following:
	\begin{itemize}
		\item An involutive set $\mathfrak{C} = \colors(P)$, where we typically write the involution as $c\mapsto c^\dagger$.
		\item A set of objects $P(c_0, c_1, \dots, c_n) \in \mathcal{V}$, one for each (possibly empty) profile in $\mathfrak{C}$.
		\item If $\sigma \in \Sigma_n^+$, maps 
		\[
			\sigma^* : P(\uc) = P(c_0, c_1, \dots, c_n) \to P(c_{\sigma(0)}, \dots, c_{\sigma(n)}) = P(\uc\sigma)
		\]
		satisfying $(\sigma')^*\sigma^* = (\sigma\sigma')^*$ and $(\id)^*=\id$.
	\end{itemize}
	A morphism $P \to Q$ between two collections consists of: 
	\begin{itemize}
		\item an involutive function $f: \colors(P) \to \colors(Q)$, and
		\item for each profile $\uc$ in $\colors(P)$, a morphism
		\[
			f_{\uc} : P(\uc) = P(c_0, c_1, \dots, c_n) \to Q(fc_0, fc_1, \dots, fc_n) = Q(f\uc)
		\]
		in $\mathcal{V}$,
	\end{itemize}
	so that for each profile $\uc$ the diagram
	\[ \begin{tikzcd}
	P(\uc) \rar{\sigma^*} \dar[swap]{f_{\uc}}& P(\uc \sigma) \dar{f_{\uc \sigma}} \\
	Q(f\uc) \rar[swap]{\sigma^*} & Q(f\uc \sigma) 
	\end{tikzcd} \]
	commutes.
\end{definition}

\begin{remark}
\label{remark collections as fibration}
Suppose that $\mathfrak{C}$ is a (possibly involutive) set.
Then there is a groupoid $\mathbf{\Sigma}_{\mathfrak{C}}^+$ whose objects are profiles $\uc = c_0, \dots, c_n$ in $\mathfrak{C}$ (including the empty profile) and whose morphisms are $\sigma : \uc \to \uc\sigma = c_{\sigma(0)}, \dots, c_{\sigma(n)}$ where $\sigma \in \Sigma_n^+$. 
A $\mathfrak{C}$-colored collection is then nothing but a functor $\mathbf{\Sigma}_{\mathfrak{C}}^+ \to \mathcal{V}$.
If $\mathfrak{C} \to \mathfrak{D}$ is a(n involutive) function, then there is an evident functor $\mathbf{\Sigma}_{\mathfrak{C}}^+ \to \mathbf{\Sigma}_{\mathfrak{D}}^+$.
A morphism of collections is precisely a natural transformation of the following form,
\[ \begin{tikzcd}
	\mathbf{\Sigma}_{\mathfrak{C}}^+  \arrow[rr] \arrow[dr, bend right, "P" swap, ""{name=B		}]  & & \mathbf{\Sigma}_{\mathfrak{D}}^+ \arrow[Rightarrow, from=B, "\eta" description] \arrow[dl, bend left, "Q"]\\
		 & \mathcal{V}
\end{tikzcd} \]
and the category of collections has an evident fibration (as in \S\ref{subsection grothendieck fibrations}) to the category of involutive sets $\iset$. The fiber over $\mathfrak{C} \in \iset$ is precisely the category of functors from $\mathbf{\Sigma}_{\mathfrak{C}}^+$ to $\mathcal{V}$.
A similar story could be told with the `restricted' collections that underlie colored operads.
\end{remark}

We now turn to our definition of colored cyclic operad. 
As our cyclic operads are always colored, we will omit the adjective and refer to \emph{monochrome} cyclic operads to refer to the uncolored setting (that is, where the color set $\mathfrak{C}$ is a singleton and the involution is necessarily trivial).
In the monochrome case, we take inspiration more from \cite[Remark II.5.10]{MarklShniderStasheff:OATP}\footnote{The first published places to make this remark precise are the non-skeletal variants which have appeared in \cite[Definition A.2]{Markl:MEOSFTNMO}, \cite[Definition 3.2]{Obradovic:MDCO}, and \cite[Definition 1.1]{ObradovicCurien:FLCO}.} 
than from \cite[Definition 2.1]{GetzlerKapranov:COCH}.
For the axioms, we will use shorthand and refer to \emph{elements} $x \in P(\uc)$ (but see Remark~\ref{remark on elements}) and for such an element $x \in P(\uc)$ we will write $\| x \| = \| \uc \|$.

\begin{definition}
\label{definition: cyclic operad}
	A \emph{cyclic operad} $P$ in $\mathcal{V}$ consists of a collection (also called $P$) as well as
	\begin{itemize}
		\item distinguished elements $\id_c : \mathbf{1}_{\mathcal{V}} \to P(c^\dagger, c)$ for each $c\in \mathfrak{C}$, and
		\item if $\uc = c_0, \dots, c_n$, $\ud = d_0, \dots, d_m$, $0\leq i \leq n$, $0\leq j \leq m$, and $c_i = d_j^\dagger$, a map
		\[
			\comp{i}{j} : P(\uc) \otimes P(\ud) \to P(\uc \comp{i}{j} \ud). 
		\]
	\end{itemize}
	These data should satisfy a short list of axioms:
\begin{enumerate}[label={C.\arabic*}),ref={C.\arabic*}]
	\item\label{cyclic swap axiom} If $\| x \| = n$ and $\| y \| = m$, we have $(\tau^{m-j+i})^*(x \comp{i}{j} y) = y \comp{j}{i} x$.
	\item\label{cyclic associativity axiom} 
	If in addition $\| z \| = p$, $0\le k\le p$, and $k\neq j$, we have
	\[
	 	(x \comp{i}{j} y) \comp{\beta_{j,i}^{m,n}(k)}{\ell} z = x \comp{i}{\alpha_{k,p}^m(j)} (y \comp{k}{\ell} z),
	\] 
	whenever the indicated compositions are defined.
	\item\label{cyclic symmetry axiom} Suppose that $\sigma_1 \in \Sigma_n^+$, $\sigma_2 \in \Sigma_m^+$, and \[ \sigma = (\alpha_{\sigma_1(i),m} \amalg \beta_{\sigma_2(j), \sigma_1(i)}) (\sigma_1 \amalg \sigma_2) (\alpha_{i,m} \amalg \beta_{j,i})^{-1}, \]
that is, $\sigma$ is defined so that the diagram 
\[ \begin{tikzcd}[column sep=3cm]	
{[0,n]\setminus \{i\}} \amalg {[0,m]\setminus \{j\}} \dar[swap]{\sigma_1 \amalg \sigma_2} \rar{\alpha_{i,m} \amalg \beta_{j,i}}
&
{[0,n+m-1]} \dar{\sigma}
\\
{[0,n]\setminus \{\sigma_1(i)\}} \amalg {[0,m]\setminus \{\sigma_2(j)\}} \rar[swap]{\alpha_{\sigma_1(i),m} \amalg \beta_{\sigma_2(j), \sigma_1(i)}}
&
{[0,n+m-1]}
\end{tikzcd}\]
commutes. Then 	 
	$(\sigma_1^*x)\comp{i}{j}(\sigma_2^*y) = \sigma^*(x \comp{\sigma_1(i)}{\sigma_2(j)} y)$.
	\item\label{cyclic identity axiom} $x \comp{i}{1} \id_c = x$ whenever this is defined, and $\tau^* \id_c = \id_{c^\dagger}$.
\end{enumerate}
Morphisms of cyclic operads are precisely the morphisms of the underlying collections which are compatible with the $\id_c$ and the $\comp{i}{j}$.
We denote the resulting category by $\cyc(\mathcal{V})$. We write $\cyc_{\mathfrak{C}}(\mathcal{V})$ for the subcategory consisting of those maps with $\colors(P) = \colors(Q) = \mathfrak{C}$, and $f = \id_{\mathfrak{C}}$. Often we will omit the ground category $\mathcal{V}$ from the notation, and write $\cyc$ instead of $\cyc(\mathcal{V})$.
\end{definition}
The axioms for a cyclic operad may seem inscrutable on first glance, but as in the case of the operad axioms, all have straightforward interpretations in terms of manipulations of grafting of trees. 
Examples are provided in Figures~\ref{figure: cyclic swap axiom}, \ref{figure: cyclic associativity axiom}, \ref{figure: cyclic symmetry axiom}, and \ref{figure: cyclic identity axiom}.
See Section~\ref{section: relation to other definitions} for some discussion of how Definition~\ref{definition: cyclic operad} relates to others in the literature. 
The above definition, when $\mathcal{V} = \set$, is essentially equivalent to the $\ast$-polycategories of \cite{Shulman:2CDCPMA,Hyland:PTA}.


\begin{figure}
\begin{center}
\begin{tikzpicture}
\begin{scope}[shift={(0:1.6)}, xscale= 88/45]
\draw[thin, red](0,0) circle(45pt);
\foreach \innerangle/\outerangle in
{
	121.218/246.702, 
	68.75/180, 
	29.849/113.298, 
	330.151/436.831, 
	-68.75/35.541, 
	238.782/324.459, 
	180.0/283.169 
}
{
	\curvemod{20}{\innerangle}{45pt}{\outerangle}{60pt}{line width=2pt, white}{0}
	\draw[thin, red](\innerangle-10:45pt)arc[radius = 45pt, start angle =\innerangle-10, end angle =\innerangle+10];
	\curvemod{20}{\innerangle}{45pt}{\outerangle}{60pt}{very thin, blue}{0}
}
\draw(0,0) circle(60pt);
\end{scope}

\node[circle, draw, line width=2, inner sep=1pt, minimum size = .43cm] (xctr) at (0,0) {$x$};
\foreach \label/\angle/\endline/\middledist/\outerlabel/\outerdist in 
{
	1/0/1/1.5/{}/1.5,
	0/90/1.35/1.45/0/2.05,
	3/180/1.49/1.6/6/2.6,
	2/270/1.35/1.55/5/2.15
}
{
	\node[circle] at (\angle+10:1) {$\label$};
	\draw (xctr) edge [red, thin] (\angle:\endline);
	\draw[thick] (xctr) edge [line width=2, line cap=round] (\angle:1.25);	
	\node[color=red] at (\angle+7.5:\middledist){$\scriptstyle \outerlabel$};
	\node[color=blue] at (\angle+5:\outerdist){$\scriptstyle \outerlabel$};
}
\node at (0:1.6) {$\comp{1}{3}$};
\begin{scope}[shift={(0:3.2)}]
\node[circle, draw, line width=2, inner sep=1pt, minimum size = .43cm] (yctr) at (0,0) {$y$};
\foreach \label/\angle/\endline/\middledist/\outerlabel/\outerdist in 
{
	3/180/1/1/{}/0/, 
	2/252/1.55/1.62/4/2.25, 
	1/324/1.34/1.55/3/2.35, 
	0/36/1.34/1.47/2/2.25, 
	4/108/1.55/1.85/1/2.4 
}
{
	\node[circle] at (\angle+10:1) {$\label$};
	\draw (yctr) edge [red, thin] (\angle:\endline);
	\draw (yctr) edge [line width=2, line cap=round] (\angle:1.25);		
	\node[color=red] at (\angle+7.5:\middledist){$\scriptstyle \outerlabel$};
	\node[color=blue] at (\angle+5:\outerdist){$\scriptstyle \outerlabel$};
}
\end{scope}
\begin{scope}[yshift = -4.5cm]
\node[circle, draw, line width=2, inner sep=1pt, minimum size = .43cm] (yctr) at (0,0) {$y$};
\foreach \label/\angle/\endline/\outerdist/\outerlabel in 
{
	3/0/1/1/{}, 
	2/72/1.55/1.64/2, 
	1/144/1.34/1.5/1, 
	0/216/1.34/1.47/0, 
	4/288/1.55/1.8/6
}
{
	\node[circle] at (\angle+10:1) {$\label$};
	\draw (yctr) edge [red, thin] (\angle:\endline);
	\draw[thick] (yctr) edge [line width=2, line cap=round] (\angle:1.25);	
	\node[color=red] at (\angle+7.5:\outerdist){$\scriptstyle \outerlabel$};
}

\draw[thin, red] (0:1.6) ellipse(88pt and 45pt);
\node at (0:1.6) {$\comp{3}{1}$};
\begin{scope}[shift={(0:3.2)}]
\node[circle, draw, line width=2, inner sep=1pt, minimum size = .43cm] (xctr) at (0,0) {$x$};
\foreach \label/\angle/\endline/\outerdist/\outerlabel in 
{
	1/180/1/1.5/{},
	0/270/1.36/1.47/5,
	3/0/1.49/1.6/4,
	2/90/1.36/1.53/3
}
{
	\node[circle] at (\angle+10:1) {$\label$};
	\draw (xctr) edge [red, thin] (\angle:\endline);
	\draw[thick] (xctr) edge [line width=2, line cap=round] (\angle:1.25);	
	\node[color=red] at (\angle+7.5:\outerdist){$\scriptstyle \outerlabel$};
}
\end{scope}
\end{scope}
\end{tikzpicture}
\end{center}
\caption{An illustration of the rule~\eqref{cyclic swap axiom}. The expressions $(\tau^2)^*(x\comp{1}{3} y)$ and $y\comp{3}{1} x$ are equal.
}\label{figure: cyclic swap axiom}
\end{figure}
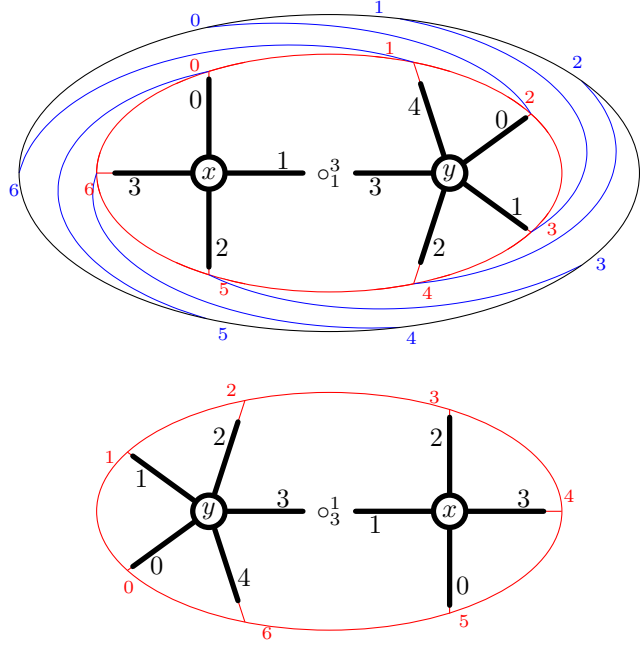


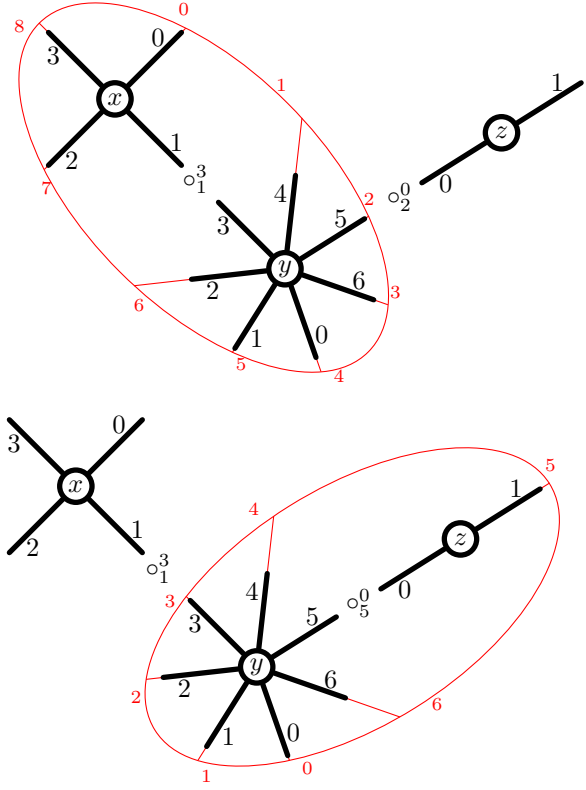
\begin{figure}
\begin{center}
\begin{tikzpicture}
\node[circle, draw, line width=2, inner sep=1pt, minimum size = .43cm] (xctr) at (0,0) {$x$};
\foreach \label/\angle/\endline/\outerdist/\outerlabel in 
{
	0/45/1.33/1.5/0,
	1/315/1/1/{},
	2/225/1.33/1.5/7,
	3/135/1.43/1.6/8
}
{
	\node[circle] at (\angle+10:1) {$\label$};
	\draw (xctr) edge [red, thin] (\angle:\endline);
	\draw[thick] (xctr) edge [line width=2, line cap=round] (\angle:1.25);	
	\node[color=red] at (\angle+7.5:\outerdist){$\scriptstyle \outerlabel$};
}

\node at (1.07,-1.07) {$\comp{1}{3}$};
\draw[rotate around={-45:(1.18,-1.18)}, red, thin] (1.18,-1.18) ellipse(88pt and 45pt);

\begin{scope}[xshift =2.26cm, yshift = -2.26cm]
\node[circle, draw, line width=2, inner sep=1pt, minimum size = .43cm] (yctr) at (0,0) {$y$};
\foreach \label/\angle/\endline/\outerdist/\outerlabel in 
{
	3/135.0/1/1/{}, 
	2/186.429/2.02/2/6, 
	1/237.857/1.32/1.4/5, 
	0/289.286/1.45/1.6/4, 
	6/340.714/1.45/1.5/3, 
	5/392.143/1.32/1.45/2, 
	4/443.571/2.02/2.45/1
}
{
	\node[circle] at (\angle+10:1) {$\label$};
	\draw (yctr) edge [red, thin] (\angle:\endline);
	\draw[thick] (yctr) edge [line width=2, line cap=round] (\angle:1.25);	
	\node[color=red] at (\angle+7.5:\outerdist){$\scriptstyle \outerlabel$};
}

\node at (392.143:1.8) {$\comp{2}{0}$};
\begin{scope}[shift = {(392.143:3.4)}]
\node[circle, draw, line width=2, inner sep=1pt, minimum size = .43cm] (zctr) at (0,0) {$z$};
\foreach \label/\angle in 
{
	0/212.143,
	1/392.143}
{
\node[circle] at (\angle+10:1) {$\label$};
\draw[thick] (zctr) edge [line width=2, line cap=round] (\angle:1.25);	
}
\end{scope}
\end{scope}

\end{tikzpicture}
\begin{tikzpicture}
\node[circle, draw, line width=2, inner sep=1pt, minimum size = .43cm] (xctr) at (0,0) {$x$};
\foreach \label/\angle in 
{
	0/45,
	1/315,
	2/225,
	3/135
}
{
	\node[circle] at (\angle+10:1) {$\label$};
	\draw[thick] (xctr) edge [line width=2, line cap=round] (\angle:1.25);	
}

\node at (-45:1.55) {$\comp{1}{3}$};
\begin{scope}[shift={(-45:3.4)}]
\node[circle, draw, line width=2, inner sep=1pt, minimum size = .43cm] (yctr) at (0,0) {$y$};
\foreach \label/\angle/\endline/\outerdist/\outerlabel in 
{
	3/135.0/1.33/1.45/3,
	2/186.429/1.48/1.65/2,
	1/237.857/1.47/1.6/1,
	0/289.286/1.32/1.5/0,
	6/340.714/2.02/2.45/6,
	5/392.143/1/1.4/{},
	4/443.571/2.02/2.1/4
}
{
	\node[circle] at (\angle+10:1) {$\label$};
	\draw (yctr) edge [red, thin] (\angle:\endline);
	\draw[thick] (yctr) edge [line width=2, line cap=round] (\angle:1.25);	
	\node[color=red] at (\angle+7.5:\outerdist){$\scriptstyle \outerlabel$};
}
\node at (392.143:1.6) {$\comp{5}{0}$};
\draw[rotate around={32:(392.143:1.5)}, red, thin] (392.143:1.5) ellipse(88pt and 45pt);
\begin{scope}[shift = {(392.143:3.2)}]
\node[circle, draw, line width=2, inner sep=1pt, minimum size = .43cm] (zctr) at (0,0) {$z$};
\foreach \label/\angle/\endline/\outerdist/\outerlabel in 
{
	0/212.143/1/1/{},
	1/392.143/1.4/1.55/5
}	
{
	\node[circle] at (\angle+10:1) {$\label$};
	\draw (zctr) edge [red, thin] (\angle:\endline);
	\draw[thick] (zctr) edge [line width=2, line cap=round] (\angle:1.25);	
	\node[color=red] at (\angle+7.5:\outerdist){$\scriptstyle \outerlabel$};
}
\end{scope}
\end{scope}

\end{tikzpicture}
\end{center}
\caption{An illustration of the rule~\eqref{cyclic associativity axiom}. The expressions $(x\comp{1}{3}y)\comp{2}{0} z$ and  $x\comp{1}{3}(y\comp{5}{0} z)$ are equal.}\label{figure: cyclic associativity axiom}
\end{figure}



\begin{figure}
\begin{center}
\begin{tikzpicture}
\node[circle, draw, line width=2, inner sep=1pt, minimum size = .43cm] (xctr) at (0,0) {$x$};
\foreach \label/\angle/\endline/\middledist/\middlelabel/\endlineb/\outerdist/\outerlabel/\outerangle in 
{
	1/0/1/1.5/{}/1/1.5/2/90,
	0/90/1.35/1.45/5/1.9/2/1/180,
	3/180/1.49/1.6/4/2.25/2.35/0/270,
	2/270/1.35/1.55/3/1.9/2.15/3/360
}
{
	\node[circle] at (\angle+10:1) {$\label$};
	\draw[thick] (xctr) edge [line width=2, line cap=round] (\angle:1.25);	
	\node[color=blue] at (\angle+5:2.15){$\scriptstyle \label$};
	\curvemod{20}{\angle}{1.25cm}{\outerangle}{2cm}{very thin, blue}{0}

}

\draw (0,0) circle(2);

\node at (2.5,0) {$\comp{1}{3}$};
\begin{scope}[shift={(0:5)}]
\node[circle, draw, line width=2, inner sep=1pt, minimum size = .43cm] (yctr) at (0,0) {$y$};
\foreach \innerangle/\outerangle in
{	
	36/-108,
	-36/-36,
	252/180,
	180/36,
	108/108
}
{
	\curvemod{20}{\innerangle}{1.25cm}{\outerangle}{2cm}{line width=2, white}{0}
	\curvemod{20}{\innerangle}{1.25cm}{\outerangle}{2cm}{very thin, blue}{0}
}

\draw (0,0) circle(2);

\foreach \label/\angle/\endline/\middledist/\middlelabel/\endlineb/\outerdist/\outerlabel in 
{
	3/180/1/1/{}/0/0/3, 
	2/252/1.55/1.62/4/2.15/2.25/2, 
	1/324/1.34/1.5/3/2/2.17/1, 
	0/36/1.34/1.47/2/2/2.1/0, 
	4/108/1.55/1.85/1/2.15/2.4/6 
}
{
	\node[circle] at (\angle+10:1) {$\label$};
	\draw (yctr) edge [line width=2, line cap=round] (\angle:1.25);	
	\node[color=blue] at (\angle+5:2.15){$\scriptstyle \label$};

}
\end{scope}
\begin{scope}[yshift = -5cm, xshift = .65cm]
\begin{scope}[shift={(0:1.6)}, xscale= 88/45]
\draw[thin, red](0,0) circle(45pt);
\foreach \innerangle/\outerangle/\straighthack in
{
	180.0/246.702/0, 
	121.218/180/0, 
	68.75/76.831/1, 
	29.849/113.298/0, 
	330.151/360-35.541/1, 
	-68.75/35.541/0, 
	238.782/283.169/0 
}
{
	\curvemod{20}{\innerangle}{45pt}{\outerangle}{60pt}{line width=2pt, white}{\straighthack}
	\draw[thin, red](\innerangle-10:45pt)arc[radius = 45pt, start angle =\innerangle-10, end angle =\innerangle+10];
	\curvemod{20}{\innerangle}{45pt}{\outerangle}{60pt}{very thin, blue}{\straighthack}
}
\draw(0,0) circle(60pt);

\end{scope}
\node[circle, draw, line width=2, inner sep=1pt, minimum size = .43cm] (xctr) at (0,0) {$x$};
\foreach \label/\angle/\endline/\middledist/\outerlabel/\outerdist in 
{
	2/0/1/1.5/{}/1.5,
	1/90/1.35/1.45/1/2.05,
	0/180/1.49/1.7/0/2.6,
	3/270/1.35/1.65/6/2.15
}
{
	\node[circle] at (\angle+10:1) {$\label$};
	\draw (xctr) edge [red, thin] (\angle:\endline);
	\draw[thick] (xctr) edge [line width=2, line cap=round] (\angle:1.25);	
	\node[color=red] at (\angle+7.5:\middledist){$\scriptstyle \outerlabel$};
	\node[color=blue] at (\angle+5:\outerdist){$\scriptstyle \outerlabel$};
}
\begin{scope}[shift={(0:3.2)}]
\node[circle, draw, line width=2, inner sep=1pt, minimum size = .43cm] (yctr) at (0,0) {$y$};
\foreach \label/\angle/\endline/\middledist/\outerlabel/\outerdist in 
{
	2/180/1/1/{}/0, 
	1/252/1.55/1.62/5/2.25, 
	0/324/1.34/1.5/4/2.3, 
	4/36/1.34/1.47/3/2.2, 
	3/108/1.55/1.85/2/2.4 
}
{
	\node[circle] at (\angle+10:1) {$\label$};
	\draw (yctr) edge [red, thin] (\angle:\endline);
	\draw (yctr) edge [line width=2, line cap=round] (\angle:1.25);	
	\node[color=red] at (\angle+7.5:\middledist){$\scriptstyle \outerlabel$};
	\node[color=blue] at (\angle+5:\outerdist){$\scriptstyle \outerlabel$};
}
\end{scope}
\node at (0:1.6) {$\comp{2}{2}$};

\end{scope}
\end{tikzpicture}
\end{center}
\caption{An illustration of the rule~\eqref{cyclic symmetry axiom}. The expressions $(\sigma_1^*x)\comp{3}{1} (\sigma_2^*y)$ and $\sigma^*(x\comp{2}{2}y)$ are equal.}\label{figure: cyclic symmetry axiom}
\end{figure}
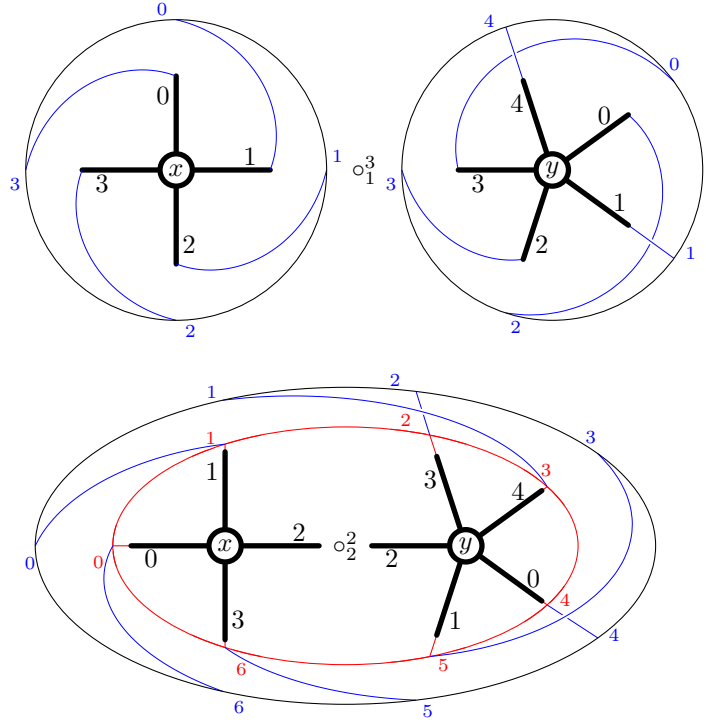



\begin{figure}
\begin{center}
\begin{tikzpicture}
\node[circle, draw, line width=2, inner sep=1pt, minimum size = .43cm] (xctr) at (0,0) {$x$};
\foreach \label/\angle in 
{
	0/120,
	1/0,
	2/240
}	
{
	\node[circle] at (\angle+10:1) {$\label$};
	\draw[thick] (xctr) edge [line width=2, line cap=round] (\angle:1.25);	
}

\node at (1.6,0) {$\comp{1}{0}$};

\begin{scope}[xshift = 3.2cm]
\node[circle, draw, line width=2, inner sep=1pt, minimum size = .43cm] (xctr) at (0,0) {$x$};
\foreach \label/\angle in 
{
	0/180,
	1/0
}	
{
	\node[circle] at (\angle+10:1) {$\label$};
	\draw[thick] (xctr) edge [line width=2, line cap=round] (\angle:1.25);	
}
\end{scope}
\begin{scope}[xshift=6.5cm]
\node[circle, draw, line width=2, inner sep=1pt, minimum size = .43cm] (xctr) at (0,0) {$x$};
\foreach \label/\angle in 
{
	0/120,
	1/0,
	2/240
}	
{
	\node[circle] at (\angle+10:1) {$\label$};
	\draw[thick] (xctr) edge [line width=2, line cap=round] (\angle:1.25);	
}
\end{scope}
\end{tikzpicture}
\end{center}
\caption{An illustration of the rule~\eqref{cyclic identity axiom}. The expressions $x\comp{1}{0}\id_c$ and $x$ are equal.}\label{figure: cyclic identity axiom}
\end{figure}
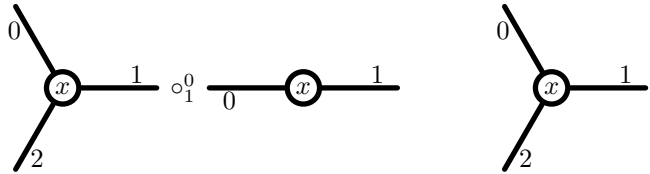


\begin{remark}\label{remark on elements}
In Definition~\ref{definition: cyclic operad}, we have listed all of the axioms as if objects in the symmetric monoidal category $\mathcal{V}$ had underlying sets, as this better represents how we think about cyclic operads.
The reader interested in the general case should have no difficulty in translating these into axioms using the structure maps of $\mathcal{V}$.
This convention could be potentially misleading: in axioms \eqref{cyclic swap axiom} and \eqref{cyclic symmetry axiom}, we use the symmetry morphism of $\mathcal{V}$. 
This could manifest if we are working in, say, the category of dg vector spaces, as \eqref{cyclic swap axiom} should read (using the Koszul sign rule)
\[
	(\tau^{m-j+i})^*(x \comp{i}{j} y) = (-1)^{ab} y \comp{j}{i} x
\] when $x\in P(\uc)_a$ and $y\in P(\ud)_b$.
\end{remark}
\begin{remark}
The involution on colors is essential to our definition. 
Monochrome cyclic operads have a trivial involution and some suggested definitions for (colored) cyclic operads follow this, using color gluing conventions for $\comp{i}{j}$ that amount to restricting our definition to the case of a trivial involution (see Section~\ref{subsec: HRY}).
Allowing general involutions yields a better-behaved category of cyclic operads, as well as better-behaved subsidiary notions.
For instance, in the formalism of \cite{HackneyRobertsonYau:HCO}, the free cyclic operad on a tree is nearly always an infinite object (see Remark 5.3 in loc.\ cit.), while in the present formalism the free cyclic operad on a tree is finite.
\end{remark}

We end this section with several examples.
These may be safely skipped, except perhaps for Example~\ref{example anti-involutive categories} which will be helpful when reading the proof of Theorem~\ref{dk model structure theorem}.

\begin{example}[Pairings]\label{example pairings}
Suppose that $\mathcal{V}$ is a closed symmetric monoidal category containing an initial object as well as objects $A$, $B$, $C$.
If $f : A \otimes B \to C$ is any fixed morphism, then there is a $\{ c, c^\dagger\}$-colored cyclic operad whose underlying collection $P$ is
\[
	P(c_0,\dots, c_n) =\begin{cases}
	A & n=0 \text{ and } c_0 = c \\
	B & n=0 \text{ and } c_0 = c^\dagger \\
	C & n=-1 \\
	\varnothing & \text{otherwise.}
	\end{cases}
\]
The map $\comp{0}{0} : P(c) \otimes P(c^\dagger) \to P(\,\,)$ is just $f$, while $\comp{0}{0} : P(c^\dagger) \otimes P(c) \to P(\,\,)$ precomposes $f$ with the symmetry isomorphism of $\mathcal{V}$. All other $\comp{i}{j}$ are the unique map from the initial object.
Thus $\cyc_{\{ c, c^\dagger\}}(\mathcal{V})$ has as a full subcategory the category of pairings in $\mathcal{V}$.
This in turn has several interesting subcategories:
\begin{itemize}
	\item Taking $C$ to be the tensor unit $\mathbf{1}_{\mathcal{V}}$, we could consider only those pairings $A \otimes B \to \mathbf{1}_{\mathcal{V}}$ whose adjoint $A \to B \multimap \mathbf{1}_{\mathcal{V}}$ is an isomorphism.
	\item Taking $A = B = \varnothing$, we see also that $\mathcal{V}$ is a full subcategory of $\cyc_{\{ c, c^\dagger\}}(\mathcal{V})$.
\end{itemize}
There are other variations, including considering the $\{ c = c^\dagger\}$-colored cyclic operads defined similarly, which accounts for symmetric pairings.
\end{example}

It follows from \eqref{cyclic swap axiom} and \eqref{cyclic identity axiom} that $\id_c \comp{1}{i} x = (\tau^{i})^*x$.
On the other hand, we could drop the identities entirely to get the following.
\begin{definition}
\label{def non-unital markl}
A \emph{non-unital Markl cyclic operad} $P$ in $\mathcal{V}$ consists of a collection (also called $P$) equipped with $\comp{i}{j}$ operations as in Definition~\ref{definition: cyclic operad} satisfying axioms~\eqref{cyclic swap axiom},~\eqref{cyclic associativity axiom}, and~\eqref{cyclic symmetry axiom}.
\end{definition}

We will use this definition only in the following example and in Proposition~\ref{prop: curien obradovich comparison}.

\begin{example}[Endomorphism cyclic operads]
Let $\mathfrak{C}$ be an involutive color set, and for every color $c$ in $\mathfrak{C}$ let $V_c$ be an object of $\mathcal{V}$. 
We write $\underline{V}$ for the collection $\{V_c\}_{c\in \mathfrak{C}}$.
We provide two formalisms to construct an ``endomorphism cyclic operad'' using $\underline{V}$ and pairing data.
\begin{enumerate}
\item First suppose given pairings $(\,,\,)_c:V_c\otimes V_{c^\dagger}\to \mathbf{1}_V$ which are symmetric in the sense that $(\,,\,)_c$ and $(\,,\,)_{c^\dagger}$ are interchanged by the symmetrizer of $\mathcal{V}$. 
Then the \emph{covariant endomorphism cyclic operad} of $\underline{V}$ with respect to these pairings is a non-unital Markl cyclic operad $\ndomorphisms^{\mathrm{co}}(\underline{V})$ which has underlying collection
\[
\ndomorphisms^{\mathrm{co}}(\underline{V})(\underline{c})=
\bigotimes_{k=0}^{\|\underline{c}\|}V_{c_k}.
\]
The $\circ_i^j$ operations are given by applying the appropriate pairing to the specified factors of the monoidal product and rearranging.
\item Instead suppose given copairings  $\Omega^c:\mathbf{1}_{\mathcal{V}}\to V_c\otimes V_{c^\dagger}$ which are symmetric in the sense that $\Omega^c$ and $\Omega^{c^\dagger}$ are interchanged by the symmetrizer of $\mathcal{V}$. 
Then the \emph{contravariant endomorphism cyclic operad} of $\underline{V}$ with respect to these copairings is a non-unital Markl cyclic operad $\ndomorphisms^{\mathrm{contra}}(\underline{V})$ which has underlying collection
\[
\ndomorphisms^{\mathrm{contra}}(\underline{V})(\underline{c})=
\left(\bigotimes_{k=0}^{\|\underline{c}\|}V_{c_k}\right)\multimap \mathbf{1}_\mathcal{V}.
\]
The $\circ_i^j$ operations are given by applying the appropriate copairing to the specified factors of the monoidal product.
\end{enumerate}
If $\mathcal{V}$ is Cartesian closed, then the monoidal unit is terminal so that \begin{enumerate}
\item in the covariant case, the data of the pairings is uniquely determined and the $\circ_i^j$ maps just project away the specified factors, and
\item in the contravariant case, the underlying collection of the endomorphism operad has a terminal object in every profile.
\end{enumerate}
In neither case does the endomorphism cyclic operad have particularly interesting structure. 

In general, if $\mathcal{V}$ is just a closed symmetric monoidal category (such as modules over a ring), the covariant and contravariant endomorphism cyclic operads differ but both can be interesting. 
Only in fairly restricted settings can the two endomorphism cyclic operads be isomorphic.
\end{example}

Every symmetric monoidal category gives rise to a colored operad. 
When the category is equipped with an appropriate dualizing object, this colored operad admits the structure of a cyclic operad, at least morally. 

\begin{example}[$\ast$-autonomous categories]\label{example sac}
A $\ast$-autonomous category (see \cite{Barr:SAC}) is a symmetric monoidal closed category $\mathcal{V}$ which has a global dualizing object $\bot$ so that the adjoint of evaluation
\begin{equation}\label{eq dagger dagger}
	a \to (a \multimap \bot)\multimap \bot
\end{equation}
is an isomorphism for all $a$. 
Writing $a^\dagger := a \multimap \bot$, this is insisting that $(a^\dagger)^\dagger \cong a$. 
A good example is the category of finite dimensional vector spaces over a field $\Bbbk$; in this case, $\bot = \Bbbk$. In fact, any compact closed category is an example of such, with $\bot = \mathbf{1}_{\mathcal{V}}$.

Since a $\ast$-autonomous category is, in particular, a symmetric monoidal category, we already know that $\mathcal{V}$ determines a colored operad $Q$ (in $\set$)\footnote{Alternatively, one may use the internal hom of $\mathcal{V}$ and consider the corresponding operad in $\mathcal{V}$. The discussion that follows will differ only in notation.} with 
\[ Q(a_1, \dots, a_n; a_0) = \mathcal{V} (a_1 \otimes \dots \otimes a_n, a_0)\] 
(using some choice of bracketing of the iterated tensor product, as in \cite[\S 3.3]{Leinster:HOHC}).
Since $a_0 \cong a_0^{\dagger\dagger}$, we have an isomorphism $Q(a_1, \dots, a_n; a_0) \cong \mathcal{V}(a_1 \otimes \dots \otimes a_n \otimes a_0^\dagger; \bot)$.
It is tempting (especially in light of Definition~\ref{definition forgetful functor}) to define a cyclic operad $P$ by
\[
	P(a_0, \dots, a_n) := \mathcal{V}(a_0 \otimes \dots \otimes a_n; \bot).
\]
This works at least in strict\footnote{Note that if \eqref{eq dagger dagger} is not an identity we do not have a strict involution on the set of colors, in contrast with Example~\ref{example anti-involutive categories}.} 
settings (when all structural isomorphisms are identities). 
It would be an interesting problem to make this precise in generality -- currently, it is well-understood how to turn the associated linearly distributive category of a $\ast$-autonomous category into a polycategory \cite{CockettSeely:WDC,CockettSeely:WDC2}.
\end{example}

The previous example extends a relationship between symmetric monoidal categories and operads to the setting of cyclic operads. 
The following example extends a relationship between \emph{ordinary} categories and operads to the setting of cyclic operads.

\begin{example}[Anti-involutive categories]\label{example anti-involutive categories}
An anti-involutive category is a (small) category $\mathcal{C}$ together with a functor $\iota : \mathcal{C}^{op} \to \mathcal{C}$ satisfying $\iota \iota^{op} = \id_{\mathcal{C}}$.
The set of objects, $\mathfrak{C}$, of $\mathcal{C}$ has an involution coming from $\iota$.
Moreover, $\mathcal{C}$ determines a $\mathfrak{C}$-colored cyclic operad whose underlying object is given by
\[
	P(c_0,\dots, c_n) = \begin{cases}
		\mathcal{C}(c_1;\iota c_0) & n=1 \\
		\varnothing & n\neq 1.
	\end{cases}
\]
The $\Sigma_1^+$-action on $P(c_0,c_1)$ sends $f: c_1 \to \iota c_0$ to $\iota f : c_0 \to \iota c_1$.
The composition
\[
	\comp{1}{0} : P(c_0, c_1) \times P(d_0,d_1) \to P(c_0,d_1)
\]
where $c_1 = \iota d_0$ is given by the usual composition $\mathcal{C}(c_1;\iota c_0) \times \mathcal{C}(d_1;\iota d_0) \to \mathcal{C}(d_1;\iota c_0)  $ in $\mathcal{C}$, while the other three compositions are forced by Definition~\ref{definition: cyclic operad}, \eqref{cyclic swap axiom} and \eqref{cyclic symmetry axiom}.
The passage from anti-involutive categories to cyclic operads mirrors the passage from categories to colored operads.
\end{example}

\section{Adjunctions between \texorpdfstring{$\operads$}{Opd} and \texorpdfstring{$\cyc$}{Cyc}}
\label{section adjunctions}

In this section we will describe a forgetful functor from cyclic operads to operads.
Under mild hypotheses on $\mathcal{V}$, this functor admits both a left and right adjoint.
The existence of a left adjoint is common in similar contexts and both the existence and construction of the left adjoint functor should be thought of as variations on a familiar theme.
On the other hand, it is rare for algebraic forgetful functors to admit a right adjoint, and the existence and construction of the right adjoint functor may be unfamiliar, even to experts (see Remark~\ref{remark: right and left adjoints} for some historical discussion of related constructions). 
We will use both adjunctions in Section~\ref{section positive model structure} to apply a criterion of our previous paper~\cite{DrummondColeHackney:CFERIQMSAAIIC}.

We begin by describing how to recover the data of a cyclic operad from a small part of it.

\begin{lemma}\label{lemma cyclic operads circ i zero}
		In any cyclic operad, we have the following equations which define the maps $\comp{i}{j}$ in terms of the maps $\comp{i}{0}$ or $\comp{0}{j}$:
\begin{align*}
x\comp{i}{j} y &= 
x \comp{i}{0} (\tau_{m+1}^j)^*y
\\
&= 
(\tau_{n+m}^{-i})^* \Big[ (\tau_{n+1}^i)^* x \comp{0}{j} y \Big] \\
&=(\tau_{n+m}^{-i})^* \Big[ (\tau_{n+1}^i)^* x \comp{0}{0} (\tau_{m+1}^j)^*y \Big]
\end{align*}
	where $n = \|x \|$ and $m= \| y\|$.
\end{lemma}
\begin{proof}
	As noted near the beginning of Section~\ref{subsection: colored cyclic operads}, $\beta_{j,i} = \beta_{0,i} \tau^{-j}$. Therefore the diagram
	\[ \begin{tikzcd}[column sep=3cm]	
{[0,n]\setminus \{i\}} \amalg {[0,m]\setminus \{j\}} \dar[swap]{\id \amalg \tau^{-j}} \rar{\alpha_{i,m} \amalg \beta_{j,i}}
&
{[0,n+m-1]} \dar{\id}
\\
{[0,n]\setminus \{i\}} \amalg {[0,m]\setminus \{ 0 \}} \rar[swap]{\alpha_{i,m} \amalg \beta_{0, i}}
&
{[0,n+m-1]}
\end{tikzcd}\]
commutes.
Thus, by \eqref{cyclic symmetry axiom} we have
\begin{align*}
	x \comp{i}{j} y &= x \comp{i}{j} \Big[ (\tau^{-j})^* (\tau^j)^* y \Big] \\
	&= x \comp{i}{0} \Big[ (\tau^j)^* y \Big], 
\end{align*}
which was the first statement.

For the second statement, we use use the first statement and \eqref{cyclic swap axiom} twice:
\begin{align*}
	(\tau^i)^* (x \comp{i}{j} y ) &= (\tau^{j+n})^* (\tau^{m-j+i})^*(x \comp{i}{j} y ) \\
	&= (\tau^{j+n})^* (y \comp{j}{i} x) \\
	&= (\tau^{j+n})^* (y \comp{j}{0} (\tau^i)^*x) \\
	&= [(\tau^i)^*x] \comp{0}{j} y.
\end{align*}
The final equality in the statement of the lemma follows from the previous two.
\end{proof}

Lemma~\ref{lemma cyclic operads circ i zero} shows us that the structure of cyclic operad is substantially overdetermined: instead of providing all of the $\comp{i}{j}$, we could just provide the $\comp{i}{0}$.
This provides a connection with the classical definitions of cyclic operads as operads with extra symmetries in each arity.
As we will see in Proposition~\ref{prop: classic cyclic operad comparison}, our cyclic operads are strictly more general than this, even in the monochrome case.
Nevertheless, we can make the following definition.

\begin{definition}[The forgetful functor]\label{definition forgetful functor}
Suppose that $P \in \cyc$.
We define an object $FP \in \operads$ as follows.
First, define $\colors(FP) \coloneqq \colors(P)$ and
\[
	FP(c_1, \dots, c_n; c_0) \coloneqq P(c_0^\dagger, c_1, \dots, c_n).
\]
Now, given biprofiles $(c_1, \dots, c_n; c_0)$ and $(d_1, \dots, d_m; d_0)$ with $d_0 = c_i$, we have 
\[ \begin{tikzcd}
FP(\uc; c_0) \otimes FP(\ud; d_0) \rar{=} \dar[swap, dashed, "\comp{i}{}", color=violet] & P(c_0^\dagger, \uc) \otimes P(d_0^\dagger, \ud) \dar{\comp{i}{0}} \\
FP(\uc \comp{i}{} \ud;c_0) \rar[swap]{=} & P(c_0^\dagger, \uc \comp{i}{} \ud) = P((c_0^\dagger, \uc) \comp{i}{0} (d_0^\dagger, \ud))
\end{tikzcd} \]
The symmetric group actions and the identities are induced from those of $P$.
\end{definition}

\begin{remark}
\label{remark: these cyclic operads have 0-0}
Given a cyclic operad $P$, there are composition operations
\begin{equation}\label{eq zero zero to minus one}
	\comp{0}{0} : P(c^\dagger) \otimes P(c) \to P(\,\,)
\end{equation}
which are not seen by the forgetful functor $F$.
That is, the underlying operad forgets more than just the extra symmetries, it also forgets about the object $P(\,\,)$ and about root-to-root gluing \emph{when both sides have no leaves}.
\end{remark}

We could modify the definition of cyclic operad to exclude the data of the object $P(\,\,)$ and the operations $\comp{0}{0}$ from \eqref{eq zero zero to minus one}. 
This is, of course, equivalent to asking that $P(\,\,)$ is a terminal object (when one exists). 
We will return to this potential modification in Section~\ref{section: relation to other definitions}.

For now, let us discuss the adjoints of $F$. 
As inspiration, consider the inclusion $\iota$ from the trivial group $\{e\}$ into $\Sigma_2$.
This inclusion induces three functors
\begin{equation}\label{cd iset set adjunctions}
\begin{tikzcd}[column sep=large]
		\set^{\Sigma_2} \rar["\iota^*" description]  & \set^{\{e\}} 
	\lar[bend right=20, "\iota_!" swap] 
	\lar[bend left=20, "\iota_*"]  \end{tikzcd}
\end{equation}
with $\iota_! \dashv \iota^* \dashv \iota_*$.
The functor $\iota^* : \iset = \set^{\Sigma_2} \to \set$ just forgets the involution.
The left adjoint is defined by 
\[ \iota_!X = \Sigma_2 \times X \] 
and the right adjoint by 
\[
	\iota_*X = X^{\Sigma_2}.
\]
For convenience, we will write the underlying set of  $\iota_!X$ as 
\[
	\{ x^a \,|\, x \in X, a \in \{0,1\} \}
\]
with $(x^0)^\dagger = x^1$ and the underlying set of $\iota_*X$ as 
$\{ (x_0, x_1) \, | \, x_i \in X \}$
with $(x_0,x_1)^\dagger = (x_1, x_0)$.
The forgetful functor $F: \cyc \to \operads$ lies over $\iota^*$.
We will now show that there are functors $L,R : \operads \to \cyc$ so that $L \dashv F \dashv R$ lies over \eqref{cd iset set adjunctions}.

\subsection{The left adjoint of the forgetful functor}
\label{subsection left adjoint}

Suppose that $P \in \operads$; we will define an object $LP \in \cyc$ as long as the ground category $\mathcal{V}$ has an initial object $\varnothing$.
Set $\colors(LP) = \iota_!\colors(P)$.
If $c_0^{a_0}, c_1^{a_1}, \dots, c_n^{a_n} = \uc^{\underline a}$ is a profile so that there exists a \emph{unique} index $k$ with $a_k = 1$, then define 
\[
	LP(\uc^{\underline a}) \coloneqq P(c_{k+1}, \dots, c_n, c_0, \dots, c_{k-1} ; c_k),
\]
otherwise $LP(\uc^{\underline a}) \coloneqq \varnothing$.
In particular, $LP(\,\,) = \varnothing$ and $LP(\uc^{\underline a}) = LP(\uc^{\underline a}\tau)$.
This leads us naturally to define $\tau^*$ to act as the identity on $LP(\uc^{\underline a})$, and the identity elements come from those in $P$.
To describe the $\comp{i}{j}$ operations, it is sufficient (by Lemma~\ref{lemma cyclic operads circ i zero}) to only define $\comp{0}{0}$.

Consider two profiles $\uc^{\underline a}$ and $\ud^{\underline b}$ so that $c_0^{a_0} = (d_0^{b_0})^\dagger$.
It is easy to define a map 
\begin{equation}\label{eq left adjoint comp ij}
 \comp{0}{0} : LP(\uc^{\underline a}) \otimes LP(\ud^{\underline b}) \to LP(\uc^{\underline a} \comp{0}{0}\ud^{\underline b})
\end{equation}
if either of the tensor factors is $\varnothing$, so we suppose that there is a unique $k_1$ with $a_{k_1} = 1$ and a unique $k_2$ with $b_{k_2} = 1$.
The compatibility condition $c_0^{a_0} = (d_0^{b_0})^\dagger$ implies that exactly one of $k_1, k_2$ is zero.
Compatibility with the symmetry, Definition~\ref{definition: cyclic operad}\eqref{cyclic swap axiom}, allows us to make the definition in just one of the cases, so we suppose that $0 < k_1 \leq n$ and $k_2=0$.
The left hand side of \eqref{eq left adjoint comp ij} is 
\[
	P(c_{k_1+1}, \dots, c_n, c_0, \dots, c_{k_1-1}; c_{k_1}) \otimes P(d_1, \dots, d_m; d_0)
\]
while the right hand side is
\[
	LP(\uc^{\underline a} \comp{0}{0}\ud^{\underline b}) = P(c_{k_1+1}, \dots, c_n, d_1, \dots, d_m, c_1, \dots, c_{k_1-1}; c_{k_1}).
\]
Then define $\comp{0}{0}$ to be the operadic composition $\circ_{n+1-k_1}$.

\begin{proposition}
\label{proposition left adjoint}
Suppose that $\mathcal{V}$ has an initial object.
With the structure from above, $LP$ is a cyclic operad. 
Moreover, the construction above is compatible with the morphisms, so constitutes a functor $L: \operads\to \cyc$.
	The functor $L$ is left adjoint to $F: \cyc \to \operads$.
\end{proposition}
\begin{proof}[Sketch of proof]
The proofs of these statements are both elementary and unpleasant so we only provide a minimal sketch.
\begin{enumerate}
\item Axiom~\eqref{cyclic swap axiom} is more or less trivial because $\tau$ acts trivially and $\circ^j_i$ is more or less symmetric.
\item Axiom~\eqref{cyclic associativity axiom} follows from the operadic associativity axiom.
\item Axiom~\eqref{cyclic symmetry axiom} follows mostly from the operadic symmetry axiom, with a special case coming from the action of $\tau$.
\item Axiom~\eqref{cyclic identity axiom} follows from the operadic identity axiom.
\item The verification of the morphism compatibility requirements is a direct computation.
\item Given a $\mathfrak{D}$-colored operad $O$, the component of the unit at $O$ is supposed to be an operad map $O\to FLO$. 
The operad $FLO$ is $\iota^*\iota_!\mathfrak{D}$-colored and at the level of colors the unit is $d\mapsto d^0$.
At the level of $\mathcal{V}$-objects, the unit is then the identity map from $O(d_1,\ldots, d_n; d_0)$ to
\[
FLO(d_1^0,\ldots,d_n^0;d_0^0)=LO(d_0^1,d_1^0,\ldots,d_n^0)=O(d_1,\ldots,d_n;d_0).
\]
Verifying that this map respects the operad structure amounts to an unfolding of the definitions.
\item Given a $\mathfrak{C}$-colored cyclic operad $P$, the component of the counit at $P$ is supposed to be a cyclic operad map $LFP\to P$.
The cyclic operad $LFP$ is $\iota_!\iota^*\mathfrak{C}$-colored and at the level of colors the counit takes $c^0$ to $c$ and $c^1$ to $c^\dagger$.
At the level of collections, the counit is then $(\tau^{-k})^*$, from
\begin{align*}
LFP(c_1^0,\ldots, c_k^1,\ldots, c_n^0)
&=FP(c_{k+1},\ldots, c_n,c_0,\ldots,c_{k-1};c_k)
\\
&=P(c_k^\dagger,c_{k+1},\ldots, c_n,c_0,\ldots, c_{k-1})
\end{align*}
to $P(c_0,\ldots, c_k^\dagger,\ldots, c_n)$.
Again it is an unfolding of definitions to verify that this map of collections respects the cyclic operad structure
\item It is a direct computation to verify the triangle identities.
\qedhere
\end{enumerate}
\end{proof}

\subsection{The right adjoint of the forgetful functor}
Suppose that $P \in \operads$; we will define an object $RP \in \cyc$ whenever the underlying category $\mathcal{V}$ has finite products.
Set $\colors(RP) \coloneqq  \iota_*\colors(P)$.
It is convenient for the moment to write $c^\bullet = (c^0, c^1)$ for an ordered pair of elements.
Set
\begin{equation}
\label{equation: right adjoint}
	RP\Big(c_0^\bullet, \dots, c_n^\bullet \Big) \coloneqq \prod_{k=0}^n P(c_{k+1}^0, \dots, c_n^0, c_0^0, \dots, c_{k-1}^0; c_k^1).
\end{equation}
Define an identity element \[ \id_{(c,c')} : \mathbf{1}_{\mathcal{V}} \to RP((c',c), (c,c')) = P(c;c) \times P(c';c') \] by $\id_c \times \id_{c'}$.
If $\sigma \in \Sigma_n^+ = \aut([0,n])$, 
define $\sigma_k = \tau^{-\sigma(k)}\sigma\tau^k$.
More explicitly, we can write $\sigma_k(r)\equiv \sigma(k+r)-\sigma(k) \pmod{n+1}$.
Since $\sigma_k (0) = 0$, we regard $\sigma_k$ as an element of $\Sigma_n$.
We then have solid arrows
\[ \begin{tikzcd}
RP\Big(c_0^\bullet, \dots, c_n^\bullet \Big) 
	\dar[swap, "\sigma^*", dashed, color=violet] 
	\rar{\pi_{\sigma(k)}} 
& 
P(c_{{\sigma(k)}+1}^0, \dots, c_{n}^0, c_{0}^0, \dots, c_{{\sigma(k)}-1}^0; c_{\sigma(k)}^1)
\dar{\sigma_k^*} 
\\
RP\Big(c_{\sigma(0)}^\bullet, \dots, c_{\sigma(n)}^\bullet \Big) 
	\rar[swap]{\pi_k} & 
 P(c_{\sigma(k+1)}^0, \dots, c_{\sigma(n)}^0, c_{\sigma(0)}^0, \dots, c_{\sigma(k-1)}^0; c_{\sigma(k)}^1)
\end{tikzcd} \]
with the indicated (co)domains. The map $\sigma^*$ is defined by $\pi_k \sigma^* \coloneqq \sigma_k^* \pi_{\sigma(k)}$.
To see that it acts as indicated, write $d_i=c^0_{\sigma(k)+i}$. 
Then the $i$th entry of $\sigma_k^*(d_1,\ldots, d_n)$ is $d_{\sigma(k+i)-\sigma(k)}=c^0_{\sigma(k+i)}$.

As in the previous section, we will only define the operations $\comp{i}{j}$ when $i=0=j$.
That is, we will define
\[
	\comp{0}{0} : RP( \underline{c^\bullet} ) \otimes RP( \underline{d^\bullet} ) \to RP( \underline{c^\bullet} \comp{0}{0}  \underline{d^\bullet} )
\]
when $(c_0^0,c_0^1) = (d_0^0,d_0^1)^\dagger = (d_0^1,d_0^0)$.
Since 
\[
	\underline{c^\bullet} \comp{0}{0}  \underline{d^\bullet} = d_1^\bullet, \dots, d_m^\bullet, c_1^\bullet, \dots, c_n^\bullet,
\]
the $k$th projection $\pi_k$ goes from $RP( \underline{c^\bullet} \comp{0}{0}  \underline{d^\bullet} )$ to
\[
	\begin{cases}
		P(d_{k+2}^0, \dots, d_m^0, c_1^0, \dots, c_n^0, d_1^0, \dots, d_k^0; d_{k+1}^1) & \text{if } 0 \leq k \leq m-1 \\
		P(c_{k-m+2}^0, \dots, c_n^0, d_1^0, \dots, d_m^0, c_1^0, \dots, c_{k-m}^0; c_{k-m+1}^1) & \text{if } m\leq k \leq n+m-1.
	\end{cases}
\]
The composite $\pi_k (\comp{0}{0})$ is defined by 
\[ \begin{tikzcd} 
RP( \underline{c^\bullet} ) \otimes RP( \underline{d^\bullet} ) \dar{swap} \\
RP( \underline{d^\bullet} ) \otimes RP( \underline{c^\bullet} ) \dar{\pi_{k+1} \otimes \pi_0} \\
P(d_{k+2}^0,\dots,d_m^0,d_0^0,\dots,d_k^0;d_{k+1}^1) \otimes P(c_1^0, \dots, c_n^0; c_0^1) \dar{\circ_{m-k}} \\
P(d_{k+2}^0,\dots,d_m^0,c_1^0, \dots, c_n^0,d_1^0,\dots,d_k^0;d_{k+1}^1)
\end{tikzcd} \]
as long as $0\leq k \leq m-1$, while if $m\leq k \leq n+m-1$ it is defined by
\[ \begin{tikzcd}
RP( \underline{c^\bullet} ) \otimes RP( \underline{d^\bullet} ) \dar{\pi_{k-m+1} \otimes \pi_0} \\
P(c_{k-m+2}^0, \dots, c_n^0, c_0^0, \dots, c_{k-m}^0; c_{k-m+1}^1) \otimes P(d_1^0, \dots, d_m^0; d_0^1) \dar{\circ_{n+m-k}} \\
P(c_{k-m+2}^0, \dots, c_n^0, d_1^0, \dots, d_m^0 ,c_1^0, \dots, c_{k-m}^0; c_{k-m+1}^1).
\end{tikzcd} \]
We have thus defined a morphism $\comp{0}{0} : RP( \underline{c^\bullet} ) \otimes RP( \underline{d^\bullet} ) \to RP( \underline{c^\bullet} \comp{0}{0}  \underline{d^\bullet} )$ in $\mathcal{V}$.

As in the previous section, a lengthy exercise gives the following.
\begin{proposition}
\label{proposition right adjoint}
Suppose that $\mathcal{V}$ admits finite products.
With the structure from above, $RP$ is a cyclic operad.
Moreover, $R$ constitutes a functor $\operads \to \cyc$ which is right adjoint to $F$. 
\end{proposition}
\begin{proof}[Sketch of proof]
We will give a briefer indication than we did in Proposition~\ref{proposition left adjoint}, simply describing the unit and counit.
Given a $\mathfrak{C}$-colored cyclic operad $P$, the unit is supposed to be a cyclic operad map $P\to RFP$, where the codomain is a $\iota_*\iota^*\mathfrak{C}$-colored cyclic operad. 
At the level of colors, the unit takes $c$ to $(c,c^\dagger)$. 
At the level of collections, the unit is then supposed to be a map from $P(c_0,\ldots, c_n)$ to
\begin{align*}
RFP((c_0,c_0^\dagger),\ldots,(c_n,c_n^\dagger))
&=\prod_{k=0}^n FP(c_{k+1},\ldots, c_n,c_0,\ldots, c_{k-1};c_k^\dagger)
\\
&=\prod_{k=0}^n P(c_k,\ldots,c_n,c_0,\ldots,c_{k-1}).
\end{align*}
The unit then has $(\tau^k)^*$ as its $k$th component.

On the other hand, given a $\mathfrak{D}$-colored operad $O$, the counit is supposed to be an operad map from $FRO$ (which is $\iota^*\iota_*\mathfrak{D}$-colored) to $O$. 
On colors, this map is projection: $(d^0,d^1)\mapsto d^0$. 
The counit map at the level of underlying $\mathcal{V}$-objects has as its domain
\[
FRO(d_1^\bullet,\ldots, d_n^\bullet;d_0^\bullet)=RO((d_0^\bullet)^\dagger,d_1^\bullet,\cdots,d_n^\bullet)
\]
which is a product whose zeroth component is  $O(d_1^0,\ldots, d_n^0;d_0^0)$. 
Then the counit is the projection onto this zeroth factor.
\end{proof}

\begin{remark}
\label{remark: right and left adjoints}
Even in the monochrome case, the right adjoint to the forgetful functor seems little-known and little-studied. It appeared in the unpublished thesis of Templeton~\cite{Templeton:SGO} for monochrome cyclic operads in sets. 
The first published description we are familiar with is the monochrome case in vector spaces or chain complexes, treated in~\cite{Ward:6OFGO}.

We have the impression that the existence and formula for the left adjoint to the forgetful functor was evident to experts from the beginning, although we have not been able to locate any early reference. 
Templeton noted that the existence of such a left adjoint (in the monochrome case in sets) follows from general principles.
\end{remark}

\section{Relation to other definitions: an interlude}
\label{section: relation to other definitions}

In this section we relate Definition~\ref{definition: cyclic operad} to several existing notions of cyclic operad:
\begin{description}
\item[Prop.~\ref{prop: classic cyclic operad comparison}] we describe a full subcategory recovering the original definition of Getzler and Kapranov~\cite{GetzlerKapranov:COCH},
\item[Prop.~\ref{prop: curien obradovich comparison}] we describe a full subcategory of non-unital Markl cyclic operads recovering the more recent variation of Curien and Obradovi\'c~\cite{CurienObradovic:CCO},
\item[Sec.~\ref{subsec: CGR}] we describe a full subcategory of a non-symmetric variant recovering the cyclic multicategories of Cheng--Gurski--Riehl~\cite{ChengGurskiRiehl:CMMAM},
\item[Sec.~\ref{subsec: HRY}] we describe a full subcategory recovering the colored cyclic operads considered by the second author and his collaborators~\cite{HackneyRobertsonYau:HCO}, and
\item[Sec.~\ref{subsec:dioperads}] we describe subcategories recovering various definitions of dioperads.
\end{description}

When comparing to existing definitions in the literature, we often need to restrict to \emph{positive} cyclic operads.
In the following definition we assume that the ground category $\mathcal{V}$ has a terminal object.
\begin{definition}\label{def cycne}
A cyclic operad $P$ is \emph{positive} if $P(\,\,)$ is terminal in $\mathcal{V}$.
We write $\cycne$ for the full subcategory of $\cyc$ consisting of the positive cyclic operads.
\end{definition}
\begin{lemma}
\label{lemma cycne is reflective in cyc}
	The category of positive cyclic operads is a reflective subcategory of cyclic operads. 
\end{lemma}
\begin{proof}
Given a cyclic operad, construct a positive cyclic operad with the same color set and the same operation objects except in the empty profile.
Induce all data except the composition to the empty profile (which is uniquely determined by terminality).
This construction is functorial and by inspection left adjoint to the inclusion functor. 
\end{proof}
We write $G: \cyc \to \cycne$ for the left adjoint reflector and $I: \cycne \to \cyc$ for the inclusion. 
Explicitly, we have that $\colors(GP) = \colors(P)$ and $GP(c_0,\dots,c_n) = P(c_0,\dots,c_n)$ for $n\geq 0$.

\begin{remark}
\label{remark cycne without terminal}
One can still define a sensible category $\cycne(\mathcal{V})$ when $\mathcal{V}$ does not have a terminal object, modifying Definition~\ref{definition: cyclic operad} to exclude the operations $\comp{0}{0} : P(c^\dagger) \otimes P(c) \to P(\,\,)$.
The functor $G: \cyc(\mathcal{V}) \to \cycne(\mathcal{V})$ will even still exist without this hypothesis, though the inclusion functor $I$ will not.
\end{remark}

Now, if $P \in \operads$, then $RP$ is a positive cyclic operad, while $LP$ is not. 
The cyclic operad $LP$ has the property that if $LP(c^a) \neq \varnothing$, then $LP((c^a)^\dagger) = \varnothing$.
This implies that the domain of $\comp{0}{0} : LP(c^a) \otimes LP((c^a)^\dagger) \to LP(\,\,)$ is always $\varnothing$, so $LP$ can be recovered from $GLP$ simply by replacing $*$ by $\varnothing$.

The cyclic operads involved in Propositions~\ref{prop: classic cyclic operad comparison},~\ref{prop: CGR comparison}, and~\ref{prop: HRY comparison} all have underlying \emph{operads}, so can't have any data in the empty profile level.
This is the reason why all three propositions deal with \emph{positive} cyclic operads.

The proofs of the propositions in the remainder of this section follow from a combination of
\begin{enumerate}
	\item concrete unpacking of the definitions in restricted situations, and 
	\item colored and ground-category independent versions of the interpolation between ``entries-only'' and ``exchangeable-output'' type definitions given explicitly in $\set$ in the monochrome case in~\cite{Obradovic:MDCO}.
\end{enumerate}
These comparisons are not related to the main thrust of the paper and are included primarily for the reader's convenience. Consequently proofs will be omitted.

\begin{proposition}
	\label{prop: classic cyclic operad comparison}
	The category of monochrome positive cyclic operads is isomorphic to the category of cyclic operads of \cite{GetzlerKapranov:COCH}.
\end{proposition}

\begin{proposition}
The category of monochrome cyclic operads is isomorphic to the category of augmented cyclic operads of \cite{HinichVaintrob:COACD}.
\end{proposition}

We also have the following comparison between our non-unital Markl cyclic operads and one of the definitions of cyclic operad from \cite{CurienObradovic:CCO}.

\begin{proposition}
\label{prop: curien obradovich comparison}
Suppose $\mathcal{V} = \set$ and $\mathfrak{C} = \{c\}$ is a one-element set (the monochrome case).
Then Definition~\ref{def non-unital markl} recovers the \emph{skeletal entries-only cyclic operads} of \cite[4.2.1]{CurienObradovic:CCO}\footnote{We note that the `$j+p-1$' appearing as an index in their {\tt (sA1)} should actually be `$j+p-2$'.}.
Further restricting Definition~\ref{def non-unital markl} to those $P$ satisfying $P(c) = P(\,\,) = \varnothing$, we recover their \emph{constant-free} skeletal entries-only cyclic operads (see the discussion in op.\ cit.\ following Definition 5).
In either case, their $\altc(m)$ corresponds to our $P(\underbrace{c,\ldots,c}_m)$.
\end{proposition}

\subsection{Small cyclic multicategories}
\label{subsec: CGR}

For our next comparison, to \cite{ChengGurskiRiehl:CMMAM}, we will introduce a non-$\Sigma$ variant of cyclic operads.
We will examine a specific case of the symmetry axiom~\eqref{cyclic symmetry axiom} where we take $\sigma_1=\tau_{n+1}$ and $\sigma_2=\id$.
In order to express the case we have in mind explicitly, we calculate:
\begin{equation}\label{eq alpha beta tau zero}
	\begin{aligned}
	\alpha_{0,m}^n\tau_{n+1} &= \tau_{n+m}^m \alpha_{n,m}^n & : [0,n-1] \to [0,n+m-1] \\
	\beta_{j,0}^{m,n} & = \tau_{n+m}^m \beta_{j,n}^{m,n} &: [0,m] \setminus \{j\} \to [0,n+m-1]
	\end{aligned}
\end{equation}
and, if $1\leq i \leq n$, then
\begin{equation}\label{eq alpha beta tau nonzero}
	\begin{aligned}
	\alpha_{i,m}^n\tau_{n+1} &= \tau_{n+m} \alpha_{i-1,m}^n & : [0,n]\setminus\{i-1\} \to [0,n+m-1] \\
	\beta_{j,i}^{m,n} & = \tau_{n+m} \beta_{j,i-1}^{m,n} &: [0,m] \setminus \{j\} \to [0,n+m-1].
	\end{aligned}
\end{equation}
Thus, the special case of axiom~\eqref{cyclic symmetry axiom} that we are considering is
\begin{equation}\label{eq non-sigma equations}
\tag{C.$3^{\text{non-}\Sigma}$}
(\tau_{n+1}^*x) \comp{i}{j} y = \begin{cases}
	\tau_{n+m}^* (x \comp{i+1}{j} y) & 0 \leq i \leq n-1 \\
	(\tau_{n+m}^m)^* (x\comp{0}{j} y) & i = n.
\end{cases}
\end{equation}

\begin{definition}
A \emph{non-$\Sigma$ collection} is defined similarly to a collection (Definition~\ref{definition: ccc}), except that $\sigma^*$ is only defined for $\sigma$ in the subgroup $\langle \tau_{n+1} \rangle \leq \Sigma_n^+$.
Likewise, a \emph{non-$\Sigma$ cyclic operad} is defined similarly to a cyclic operad (Definition~\ref{definition: cyclic operad}) except that it only has an underlying non-$\Sigma$ collection and, furthermore, \eqref{cyclic symmetry axiom} is replaced with \eqref{eq non-sigma equations}.
\end{definition}

\begin{proposition}
	\label{prop: CGR comparison}
	The category of positive non-$\Sigma$ cyclic operads is isomorphic to the category of small cyclic multicategories of \cite[Definition 3.3]{ChengGurskiRiehl:CMMAM}.
\end{proposition}

\begin{remark}[Adjunctions]
\label{remark non-sigma adjoints}
All of the functors from \S\ref{section adjunctions} restrict to adjunctions between non-$\Sigma$ operads and non-$\Sigma$ cyclic operads.
This is the primary reason for the choice of order of entries in the definition of the underlying collection of $LP$ in \S\ref{subsection left adjoint}.
In any case, the only time the symmetric group actions come into play in the proofs are that the counit of $L \dashv F$ (in the proof of Proposition~\ref{proposition left adjoint}) and the unit of $F \dashv R$ (in the proof of Proposition~\ref{proposition right adjoint}) make use of $\tau^*$.
\end{remark}

This notion has been extended from multicategories to symmetric multicategories.

\begin{proposition}
The category of positive cyclic operads is isomorphic to the category of small cyclic symmetric multicategories of \cite{Shulman:2CDCPMA}.
\end{proposition}

\subsection{Colored exchangeable-output cyclic operads}
\label{subsec: HRY}
Suppose that $\mathfrak{C}$ is a set, which we consider as an involutive set with the trivial involution $c^\dagger = c$.
A $\mathfrak{C}$-colored cyclic operad $O$ in the sense of \cite{HackneyRobertsonYau:HCO} gives rise to a $\mathfrak{C}$-colored cyclic operad $P$ in our sense.
The underlying collection of $P$ is given by
\[
	P(c_0, \dots, c_n) = \begin{cases}
		O(c_1, \dots, c_n; c_0) & \text{ if $n \geq 0$ } \\
		* & \text{ if $n=-1$.}
	\end{cases}
\]
The $\Sigma_n^+$ action and identity elements of $P$ are the ones already held by $O$.
We need to define operations 
\[
	\comp{i}{j} : P(\uc) \otimes P(\ud) \to P(\uc \comp{i}{j} \ud)
\]
for $0 \leq i \leq n = \| \uc \|$, $0\leq j \leq m = \| \ud \|$.
For $1 \leq i \leq n$, we let $\comp{i}{0} \coloneqq \comp{i}{}$.
Further, we define
\[
	x\comp{0}{0} y = \begin{cases}
		\tau_{m+n}^* \left( \left( \tau_{n+1}^{-1} \right)^* x \comp{1}{} y\right) & \text{ if $n > 0$,} \\
		(\tau_{m+n}^{n+1})^* \left( \left( \tau_{m+1}^{-1} \right)^* y \comp{1}{} x\right) & \text{ if $n = 0$ and $m > 0$,} \\
		* & \text{ if $n=m=0$.}
	\end{cases}
\]
For other values of $j$, use the formula from Lemma~\ref{lemma cyclic operads circ i zero}.

This process is reversable, and yields the following equivalence of categories.

\begin{proposition}
\label{prop: HRY comparison}
Consider the full subcategory of $\cycne$ consisting of those positive cyclic operads $P$ so that the involution on $\colors(P)$ is trivial. 
This subcategory is equivalent to the category of colored cyclic operads of \cite{HackneyRobertsonYau:HCO}.
\end{proposition}

\subsection{Dioperads}
A dioperad is usually thought of as parametrizing algebraic structures that have multiple inputs and multiple outputs which are not interchangeable. 
The governing type of graph is then some version of directed trees. 
A priori this seems different from the undirected trees that govern cyclic operads.
Gan suggested that there might be a way to relate dioperads to cyclic operads. 
A judicious choice of color set for cyclic operads realizes this possibility, as reflected in the following proposition. 
To our knowledge no similar explicit statement appears in the literature, although Kaufmann and Lucas~\cite[6.4.1]{KaufmannLucas:DFC} treat it implicitly and it may have been known in other contexts.
\label{subsec:dioperads}
\begin{proposition}
\label{prop: dioperad comparison}
The category of cyclic operads with color set $\iota_!\mathfrak{C}$ is equivalent to the category of $\mathfrak{C}$-colored dioperads~\cite[11.5.1]{YauJohnson:FPAM}. 
Given such a cyclic operad $P$, we define the underlying $\Sigma^\mathcal{V}_S$-module of the dioperad $D\binom{\underline{d}}{\underline{c}}$ as 
\[
D\binom{\underline{d}}{\underline{c}} =P(c_1^0,\ldots,c_k^0,d_1^1,\ldots,d_\ell^1).
\]
Similarly, given a dioperad $D$, we define the underlying collection of the cyclic operad $P$ in profile $\underline{c}$ as $D\binom{\underline{c}''}{\underline{c}'}$ where $\underline{c}'$ consists of those elements of $\underline{c}$ of the form $c_i^0$ and $\underline{c}''$ consists of those elements of $\underline{c}$ of the form $c_i^1$.

More specifically, the category $\cyc_{\iota_!\{*\}}(\mathcal{V})$ is equivalent to the category of monochrome \emph{dioperads}.
It is instructive to rename the elements of the color set to $\iota_!\{*\} = \{*^0, *^1\} \cong \{\mathrm{in},\mathrm{out}\}$ (with free involution).
To match the original definition~\cite[1.1]{Gan:KDD} of a dioperad, we should further restrict to $\{\mathrm{in},\mathrm{out}\}$-colored cyclic operads $P$ such that $P(\uc^{\underline a})=\varnothing$ unless there exist at least one $\mathrm{in}$ index and at least one $\mathrm{out}$ index in $\uc^{\underline{a}}$.
\end{proposition}

\begin{remark}
\label{remark full dioperads category}
The preceding proposition indicates how the category of dioperads (for varying color sets, as in \cite[\S 2]{HackneyRobertsonYau:SMIP}) sits naturally inside of $\cyc$.
We have a commutative diagram of functors
\[ \begin{tikzcd}
	\mathbf{Dioperads}_{\mathfrak{C}} \rar{\simeq} \dar & \cyc_{\iota_!\mathfrak{C}} \dar \\
	\mathbf{Dioperads}\dar \rar  & \cyc \dar \\
	\set \rar{\iota_!} & \iset
	\end{tikzcd} \]	
whose middle arrow can be considered as a (non-full) subcategory inclusion.
Since $\operads$ is a (full) subcategory of $\mathbf{Dioperads}$, we can thus also regard $\operads$ as a subcategory of $\cyc$.
The composite $\operads \to \mathbf{Dioperads} \to \cyc$ is isomorphic to the functor $L$ from Section~\ref{subsection left adjoint}.
\end{remark}

\section{Bifibrations and colored cyclic operads}
\label{section bifibrations}
In this section, we establish bicompleteness of the category of cyclic operads given bicompleteness of the ground category. 
The route we have chosen uses the technology of Grothendieck fibrations to pass from the fixed-color case to the general case.
The reader comfortable with the existence of limits and colimits may safely skip this section.

\subsection{A quick recollection of Grothendieck fibrations}
\label{subsection grothendieck fibrations}
Suppose that $p: \mathcal{E} \to \mathcal{B}$ is a functor.
A morphism $\phi : d\to e$ of $\mathcal{E}$ is \emph{Cartesian} if for each object $u\in \mathcal{E}$, the function
\begin{equation}\label{eq: Cartesian morphism}
\begin{aligned}
	\mathcal{E}(u,d) &\to \mathcal{E}(u,e) \times_{\mathcal{B}(pu,pe)} \mathcal{B}(pu,pd) \\
	\gamma &\mapsto (\phi\gamma, p(\gamma))
\end{aligned}
\end{equation}
is a bijection. 
Dually, $\phi$ is called \emph{coCartesian} if the function 
\[ 
\mathcal{E}(e,u) \to \mathcal{E}(d,u) \times_{\mathcal{B}(pd,pu)} \mathcal{B}(pe,pu) 
\] 
is a bijection for every $u$.
The functor $p$ is a \emph{(Grothendieck) fibration} if for each object $e \in \mathcal{E}$ and each morphism $\alpha$ in $\mathcal{B}$ with codomain $p(e)$, there exists a Cartesian morphism $\phi : \alpha^*(e) \to e$ with $p(\phi) = \alpha$.
As this Cartesian morphism $\phi$ is necessarily unique up to unique isomorphism, the notation $\alpha^*(e)$ for the domain of the chosen $\phi$ is relatively harmless.
Dually, the functor $p$ is an \emph{opfibration} if for each object $d\in \mathcal{E}$ and each morphism $\alpha$ in $\mathcal{B}$ with domain $p(d)$, there exists a coCartesian morphism $\phi : d \to \alpha_!(d)$ so that $p(\phi) = \alpha$.
A \emph{bifibration} is a functor which is both a fibration and an opfibration.

If $p: \mathcal{E} \to \mathcal{B}$ is a fibration, 
then for every morphism $\alpha : a \to b$ in $\mathcal{B}$, there is an induced functor $\alpha^* : p^{-1}(b) \to p^{-1}(a)$.
To define this functor, first \emph{choose} Cartesian lifts $\phi_x : \alpha^*x \to x$ for each object $x\in p^{-1}(b)$.
On objects, $\alpha^*$ sends $x$ to the domain of $\phi_x$.
On morphisms, the functor sends $\beta : x \to y$ (with $p(\beta) = \id_{b}$) to the inverse image of $(\beta\phi_x, \id_a) \in \mathcal{E}(\alpha^*x,y) \times_{\mathcal{B}(a,b)} \mathcal{B}(a,a)$ under the bijection \eqref{eq: Cartesian morphism} associated to $\phi_y$.
The assignment $b\mapsto p^{-1}(b)$ constitutes a \emph{pseudo}functor $\mathcal{B}^{op} \to \cat$, as we had to make a choice of the Cartesian lifts.
This assignment, which takes a fibration and produces a pseudofunctor, is part of an equivalence of 2-categories between pseudofunctors $\mathcal{B}^{op} \to \cat$ and fibrations over $\mathcal{B}$.
The reverse direction is the \emph{Grothendieck construction}.
If $\Psi : \mathcal{B}^{op} \to \cat$ is a pseudofunctor, then there is a category $\int \Psi$ with objects those pairs $(b,x)$ with $b$ an object of $\mathcal{B}$ and $x$ and object of $\Psi(b)$.
A morphism from $(b,x) \to (b',x')$ consists of a pair $u : b\to b'$ and $f: x \to u^*(x') = \Psi(u)(x')$.
The evident functor $p: \int \Psi \to \mathcal{B}$ is a fibration, with $p^{-1}(b) = \Psi(b)$; this fibration is commonly referred to as the Grothendieck construction.

Likewise, pseudofunctors $\mathcal{B} \to \cat$ are essentially the same things as opfibrations over $\mathcal{B}$.
Bifibrations over $\mathcal{B}$ are the same things as pseudofunctors from $\mathcal{B}$ into a 2-category of adjunctions.

\subsection{The category of colored cyclic operads}\label{ss: category of colored cyclic operads}

If $\mathfrak{C}$ is a set (possibly equipped with involution), let $\mathbf{M}\mathfrak{C}$ denote the set\footnote{Notice that $\mathbf{M}\mathfrak{C}$ is precisely the set of objects of the category $\mathbf{\Sigma}_{\mathfrak{C}}^+$, briefly introduced in Remark~\ref{remark collections as fibration}.} of $\mathfrak{C}$ profiles, that is, the set of finite ordered lists whose elements are chosen from $\mathfrak{C}$.
If $\mathfrak{C}$ is a set equipped with an involution $\dagger$, then there is an $\mathbf{M}\mathfrak{C}$-colored operad $\mathscr{C}_{\mathfrak{C}}$ (in $\set$) so that $\mathscr{C}_{\mathfrak{C}}$ algebras are precisely the $\mathfrak{C}$-colored cyclic operads. The category $Alg(\mathscr{C}_{\mathfrak{C}})$ is equivalent to the category $\cyc_{\mathfrak{C}}(\set)$ from above.
We can also consider the colored operad $\mathscr{C}_{\mathfrak{C}}^{\mathcal{V}}$ in $\mathcal{V}$, which is obtained from $\mathscr{C}_{\mathfrak{C}}$ by using the unique cocontinuous symmetric monoidal functor $\set \to \mathcal{V}$; in this case we again have $Alg(\mathscr{C}_{\mathfrak{C}}^{\mathcal{V}}) \simeq \cyc_{\mathfrak{C}}(\mathcal{V}).$

The operad $\mathscr{C}_{\mathfrak{C}}$ is generated by binary operations 
\[ \comp{i}{j} \in \mathscr{C}_{\mathfrak{C}}(\uc , \ud; \uc \comp{i}{j} \ud)
\]
when $c_i = d_j^\dagger$, unary operations
\[
	\sigma^* \in \mathscr{C}_{\mathfrak{C}}(\uc; \uc\sigma)
\]
when $\sigma \in \Sigma_{\|\uc\|}^+$, and nullary operations
\[
	\id_c \in \mathscr{C}_{\mathfrak{C}}(\,\,; (c^\dagger, c)).
\]
If $\alpha: \mathfrak{C} \to \mathfrak{D}$ is a morphism in $\iset$, then there is a map of colored operads
$\mathscr{C}_{\mathfrak{C}} \to \mathscr{C}_{\mathfrak{D}}$ which has color map $\mathbf{M}\alpha : \mathbf{M}\mathfrak{C} \to \mathbf{M}\mathfrak{D}$ and sends generators to generators.
The assignment $\mathfrak{C} \mapsto \mathscr{C}_{\mathfrak{C}}^{\mathcal{V}}$ thus constitutes a functor $\iset \to \operads(\mathcal{V})$.

\begin{lemma}
	Let $\mathcal{V}$ be a cocomplete closed symmetric monoidal category. 
	Suppose that $F : \mathcal{B} \to \operads(\mathcal{V})$ is any functor and let $p : \mathcal{E} \to \mathcal{B}$ 
	be the Grothendieck fibration associated to the composite \[ \mathcal{B}^{op} \xrightarrow{F^{op}} \operads(\mathcal{V})^{op} \xrightarrow{Alg(-)} \cat.\]
	Then $p$ is also an opfibration.
\end{lemma}
\begin{proof}
	If $f: O \to O'$ is any map of colored operads, cocompleteness of $\mathcal{V}$ assures that the map $f^* : Alg(O') \to Alg(O)$ admits a left adjoint $f_!$.
	This is enough to ensure that $p$ is an opfibration as well (see, for instance, \cite[Lemma 9.1.2]{Jacobs:CLTT}).
\end{proof}

\begin{lemma}
\label{lemma bifibration}
If $\mathcal{V}$ is cocomplete, then the functor $\cyc(\mathcal{V}) \to \iset$ which sends a cyclic operad to its involutive set of colors is a Grothendieck bifibration.
\end{lemma}
\begin{proof}
Apply the previous lemma to the functor $\mathfrak{C} \mapsto \mathscr{C}_{\mathfrak{C}}^{\mathcal{V}}$ from $\iset$ to $\operads(\mathcal{V})$. The morphisms from Definition~\ref{definition: cyclic operad} are precisely the morphisms in the Grothendieck construction associated to the contravariant functor $\mathfrak{C} \mapsto Alg(\mathscr{C}_{\mathfrak{C}}^{\mathcal{V}}) \simeq \cyc_{\mathfrak{C}}(\mathcal{V})$  from $\iset$ to $\cat$.
\end{proof}

\begin{proposition}\label{cyc is bicomplete}
Suppose that $\mathcal{V}$ is a closed symmetric monoidal category. 
If $\mathcal{V}$ is bicomplete, then so is $\cyc(\mathcal{V})$.
\end{proposition}
\begin{proof}
If $p : \mathcal{E} \to \mathcal{B}$ is any opfibration with $\mathcal{B}$ and each fiber $\mathcal{E}_b$ cocomplete, then $\mathcal{E}$ is cocomplete as well. The dual statement for fibrations and completeness also holds.  This is classical -- see Exercise 9.2.4, p.531 of \cite{Jacobs:CLTT}; a proof appears of the first statement appears in Remark 2.4.3 and Proposition 2.4.4 of \cite{HarpazPrasma:GCMC}.
In our case, we know that $\cyc_{\mathfrak{C}}(\mathcal{V})$ is cocomplete and complete since it is algebras over an operad,  hence the result follows from Lemma~\ref{lemma bifibration} and the fact that $\iset$ is bicomplete.
\end{proof}

This proof goes through equally well for $\cycne(\mathcal{V})$, as long as one uses the variation on $\mathscr{C}_{\mathfrak{C}}$ which omits the color $(\,\,)$ from $\mathbf{M}\mathfrak{C}$ along with the associated generators $\comp{0}{0} \in \mathscr{C}_{\mathfrak{C}}((c^\dagger), (c); (\,\,))$ (see Remark~\ref{remark cycne without terminal}). 
Alternatively, since the subcategory $\cycne(\mathcal{V})$ of $\cyc(\mathcal{V})$ is reflective, $\cycne(\mathcal{V})$ has all limits and colimits that $\cyc(\mathcal{V})$ has (\cite[Proposition 4.5.15]{Riehl:CTC}).

\section{The positive model structure}\label{section positive model structure}

In this final section we establish the main result of the paper: the existence of a Dwyer--Kan type model category structure on the category of positive cyclic operads enriched in $\sset$.
We also conjecture that this model structure can be extended to the category of \emph{all} cyclic operads, and give some indication of how one might proceed. 

Suppose that $f: P \to Q$ is a map in $\cyc(\mathcal{V})$ (resp.\ in $\operads(\mathcal{V})$). If X is a property of maps of $\mathcal{V}$, we say that $f$ satisfies X locally if for each list of colors $c_0, c_1, \dots, c_n \in \colors(P)$, the induced map
\[
	P(c_0, \dots, c_n) \to Q(fc_0, \dots, fc_n) \,\, \Big( \text{resp.\ } P(c_1, \dots, c_n; c_0) \to Q(fc_1, \dots, fc_n; fc_0) \Big)
\]
satisfies X. 
For example, if $\mathcal{V}$ is a model category, we can refer to $f$ being locally a fibration or locally a weak equivalence.

We will momentarily introduce a model structure on the category $\cyc(\sset)$ whose objects are cyclic operads enriched in simplicial sets.
Let $U: \operads(\set) \to \cat$ be the `underlying category' functor.
We have a diagram
\[
\begin{tikzcd}
	\cyc(\sset) \rar{F} \dar{\pi_0} & \operads(\sset) \rar{U} \dar{\pi_0} & \cat(\sset) \dar{\pi_0} \\
	\cyc(\set) \rar[swap]{F} & \operads(\set) \rar[swap]{U} & \cat 
\end{tikzcd}
\]
where the vertical arrows are induced from the path components functor $\pi_0 : \sset \to \set$. 
Write $[-] : \cyc(\sset) \to \cat$ for the composite of this diagram, and use the same notation for the functor $\operads(\sset) \to \cat$ (in particular, this means that $[-] = [F(-)]$, which we hope causes no confusion).

\begin{definition}\label{def dk and fibrations}
Suppose that $f: P \to Q$ is a map in $\cyc(\sset)$ or $\operads(\sset)$.
We say that $f$ is a \emph{Dwyer--Kan equivalence} (resp.\ \emph{isofibration}) if 
\begin{itemize}
	\item $f$ is locally a weak equivalence (resp.\ Kan fibration) of simplicial sets, and
	\item the functor $[f]: [P] \to [Q]$ is an equivalence of categories (resp.\ isofibration of categories).
\end{itemize}
If $f: P \to Q$ is a map in $\cyc(\sset)$, we say that $f$ is a \emph{positive Dwyer--Kan equivalence} (resp.\ \emph{positive isofibration}) if $G(f) : G(P) \to G(Q)$ is a Dwyer--Kan equivalence (resp.\ an isofibration).
\end{definition}

Here, $G : \cyc(\sset) \to \cycne(\sset)$ is the reflection from Lemma~\ref{lemma cycne is reflective in cyc}.
The only difference between a map $f: P \to Q$ being a Dwyer--Kan equivalence or a positive Dwyer--Kan equivalence is that in the latter case there is no restriction at all on the map $P(\,\,) \to Q(\,\,)$.

\begin{remark}[On Dwyer--Kan equivalences]
Let us make a philosophical comment.
In the case that we are working in situations where all maps $f$ are bijective on color sets, it is relatively clear how one should define the notion of equivalence: it should simply be a local weak equivalence.
In the general setting, we should insist that our equivalences induce bijections on \emph{``homotopy classes''} of colors.
Being homotopic should be a binary relationship, rather than a many-to-many or many-to-one relationship, which is why we can detect it at the categorical level.
Indeed, the idea of homotopy classes of colors is made precise by looking at the isomorphism classes of objects of $[P]$, and the second condition from Definition~\ref{def dk and fibrations} ensures that we have a bijection between the isomorphism classes of objects of $[P]$ and $[Q]$. (Note that the first condition actually implies that $[f]$ is fully-faithful).
\end{remark}

\begin{theorem}\label{dk model structure theorem}
	The category $\cycne(\sset)$ admits a model structure with isofibrations as the fibrations and with Dwyer--Kan equivalences as the weak equivalences.
	The category $\cyc(\sset)$ admits a model structure whose fibrations are positive isofibrations and whose weak equivalences are positive Dwyer--Kan equivalences.
\end{theorem}
\begin{proof}
We have already established bicompleteness of these categories in the previous section (see Proposition~\ref{cyc is bicomplete}).

	We first prove the second statement of the theorem.
	Recall from \cite[Theorem 1.14]{CisinskiMoerdijk:DSSO} that $\operads(\sset)$ admits a model structure with Dwyer--Kan equivalences as the weak equivalences and the isofibrations as the fibrations.
	By \cite[Theorem 2.3]{DrummondColeHackney:CFERIQMSAAIIC}, it is enough to show that the composite $FR: \operads(\sset) \to \cyc(\sset) \to \operads(\sset)$ is a right Quillen functor.
	Suppose that $f$ is locally a fibration. Since fibrations are stable under products, the map
	\begin{equation}\label{eq FRP to FRQ}
		FRP(c_1^\bullet, \dots, c_n^\bullet; (c_0^\bullet)^\dagger) \to FRQ(fc_1^\bullet, \dots, fc_n^\bullet; (fc_0^\bullet)^\dagger)
	\end{equation}
	is a fibration (see Equation~\eqref{equation: right adjoint} for the formula for $R$).
	Likewise, trivial fibrations are stable under products, so if $f$ is locally a trivial fibration then \eqref{eq FRP to FRQ} is again a trivial fibration.
	Thus $FRf$ is locally a (trivial) fibration if $f$ is a (trivial) fibration.

	We now turn from the local structure to the categorical structure.
	The diagram
	\[
	\begin{tikzcd}
		\operads(\sset) \rar{R} \dar[swap]{\pi_0} & \cyc(\sset) \rar{F} \dar{\pi_0} & \operads(\sset) \dar{\pi_0} \\
		\operads(\set) \rar{R} \dar[swap]{U} & \cyc(\set) \rar{F} \dar{U} & \operads(\set) \dar{U} \\
		\cat \rar[swap]{R} & \mathbf{iCat} \rar[swap]{F} & \cat
	\end{tikzcd}
	\]
	commutes, where the maps on the bottom line are those from \cite[2.7]{DrummondColeHackney:CFERIQMSAAIIC}.
	We thus have $[FR(-)] = U\pi_0FR(-) = FRU\pi_0(-) = FR[-]$.
	If $f$ is a (trivial) fibration in $\operads(\sset)$, then $[f]$ is a (trivial) fibration in $\cat$. As we saw in \cite[2.7]{DrummondColeHackney:CFERIQMSAAIIC}, this implies that $FR[f] = [FR(f)]$ is a (trivial) fibration in $\cat$.
	Thus if $f$ is a (trivial) fibration, so is $FR(f)$. The characterization of weak equivalences and fibrations on $\cyc(\sset)$ follows immediately, since these are both created by $F$.

	The proof of the existence of the model structure on $\cycne(\sset)$ follows exactly as in the previous paragraphs, using the adjunctions $GL \dashv FI \dashv GR$ in place of $L \dashv F \dashv R$ since $FIGR = FR$.
	\end{proof}

Definition~\ref{def dk and fibrations} can be adapted to other monoidal model categories $\mathcal{V}$.
The main point is to
\begin{itemize}
\item change the local condition to the suitable condition in $\mathcal{V}$, and 
\item to replace the functor $\pi_0 : \sset \to \set$ with $X \mapsto Ho(\mathcal{V})(\mathbf{1}, X)$; this yields a functor from $\mathcal{V}$-enriched categories to ordinary categories $\cat(\mathcal{V}) \to \cat$.
\end{itemize} 
The second point means that we have functors $[-] : \cyc(\mathcal{V}) \to \cat$ and $[-]: \operads(\mathcal{V}) \to \cat$.

In nice situations, one will have a model structure on $\operads(\mathcal{V})$ with Dwyer--Kan equivalences as the weak equivalences.
Existence of this model structure is a major topic of \cite{Caviglia:MSECO,Yau:DKHTAOC}, with ideas going back to the setting of $\cat(\mathcal{V})$ from \cite{Muro:DKHTEC,BergerMoerdijk:HTEC}.
We will not dwell on the technical requirements here, and instead say that $\mathcal{V}$ is \emph{DK-admissible} if a model structure exists on $\operads(\mathcal{V})$ with weak equivalences and fibrations as in Definition~\ref{def dk and fibrations}.
The proof of the following theorem is essentially identical to that of Theorem~\ref{dk model structure theorem}.

\begin{theorem}
If $\mathcal{V}$ is DK-admissible, then $\cycne(\mathcal{V})$ admits a model structure with weak equivalences and fibrations as in Definition~\ref{def dk and fibrations}. \qed
\end{theorem}

Likewise, the category of small cyclic multicategories (of \cite{ChengGurskiRiehl:CMMAM}, see Proposition~\ref{prop: CGR comparison}) enriched in simplicial sets admits a proper model structure.

\begin{theorem}
The category of non-$\Sigma$ positive cyclic operads in $\sset$ admits a \emph{proper} model structure with isofibrations as the fibrations and with Dwyer--Kan equivalences as the weak equivalences.
\end{theorem}
\begin{proof}
As we observed in Remark~\ref{remark non-sigma adjoints}, the forgetful functor from non-$\Sigma$ positive cyclic operads to non-$\Sigma$ operads admits both adjoints, which are given by the same formulas as in the symmetric case.
The proof of Theorem~\ref{dk model structure theorem} goes through with the appropriate minor modifications.
That this model structure is proper follows from \cite[Corollary 8.9]{CisinskiMoerdijk:DSSO}, \cite[Proposition 2.4]{DrummondColeHackney:CFERIQMSAAIIC}, and the fact that this is a right-induced model structure.
\end{proof}

\subsection{Future directions}

In Theorem~\ref{dk model structure theorem} we showed that $\cyc(\sset)$ admits a model structure with weak equivalences the \emph{positive} Dwyer--Kan equivalences; let us call this the \emph{positive} model structure.
This required relatively little work to prove, as we were able to lift this model structure from the Cisinski-Moerdijk model structure on $\operads(\sset)$ using the criterion from \cite{DrummondColeHackney:CFERIQMSAAIIC}.
Most earlier definitions of cyclic operads restrict themselves to the positive case (see \S\ref{section: relation to other definitions}).
However, we still find this situation somewhat unsatisfying---we would like to have a model structure with weak equivalences the Dwyer--Kan equivalences.

\begin{conjecture}\label{better dk theorem}
	The category $\cyc(\sset)$ admits a model structure with fibrations the isofibrations and with weak equivalences the Dwyer--Kan equivalences.
\end{conjecture}

One can interpret this proposed model structure as an \emph{intersection} of model structures (in the sense of \cite[Definition 8.5]{IntermontJohnson:MSCES}).
The first model structure involved is the positive model structure from the previous section. The second model structure involved is the one lifted along the adjunction 
$\numberfunctor : \sset \rightleftarrows \cyc(\sset) : E$ where $EP = P(\,\,)$, $\colors(\numberfunctor X) = \varnothing$, and $\numberfunctor X(\,\,) = X$.
Unfortunately, recognizing that this should be an intersection of model structures is not much help in actually proving that the stated model structure exists (the hypotheses of \cite[Proposition 8.7]{IntermontJohnson:MSCES} are often difficult to check, including in this instance). 

This perspective does, however, give a sensible choice for the generating (trivial) cofibrations of the conjectural model structure.
Most aspects of a partial proof of Conjecture~\ref{better dk theorem} are formal, with one exception.
We thus make the following conjecture, which is the key missing component to establishing the previous conjecture.

\begin{conjecture}\label{pushout conjecture}
	Suppose that $\{x\} \to \mathcal{H}$ is in the set (A2) from \cite[p.2046]{Bergner:MCSCSC}.
	If
	\[ \begin{tikzcd}
	L\{x\} \dar \rar & P \dar \\
	L\mathcal{H} \rar & Q
	\end{tikzcd} \]
	is a pushout in $\cyc(\sset)$, then $P\to Q$ is a Dwyer--Kan equivalence.
\end{conjecture}
Since these maps $L\{x\} \to L\mathcal{H}$ are contained in the set of generating trivial cofibrations for the positive model structure, we already know that $P \to Q$ is a positive Dwyer--Kan equivalence.
In order to prove Conjecture~\ref{pushout conjecture}, it thus suffices to show that $P(\,\,) \to Q(\,\,)$ is a weak equivalence of simplicial sets.
We do not know of any shortcut to establish this weak equivalence.

\begin{remark}[Towards a dendroidal model]
In analogy with the Moerdijk--Weiss category $\Omega$ from \cite{MoerdijkWeiss:DS}, the category of positive cyclic operads admits a full subcategory $\Omega_{cyc}^{\Sigma}$ of \emph{unrooted trees}.
The nonsymmetric version (for planar unrooted trees), $\Omega_{cyc}$, is explained in the work of Tashi Walde \cite{Walde:2SSIIO}, while the symmetric version appears in work of the second author, Robertson, and Yau \cite{HackneyRobertsonYau:GCHMO}. 
The category $\Omega_{cyc}^{\Sigma}$ admits a bijective-on-objects full functor to the category $\Xi$ from \cite{HackneyRobertsonYau:HCO}.
Further, there is a functor $f : \Omega \to \Omega_{cyc}^{\Sigma}$ which just forgets the root of each tree.
We conjecture that the category of simplicial-set valued presheaves on $\Omega_{cyc}^{\Sigma}$ admits a Rezk-type (that is, `complete Segal space'-type) model structure, which is equivalent to the model structure from Theorem~\ref{dk model structure theorem}.
This model structure should be lifted along $f^*$
\[
\begin{tikzcd}[column sep=large]
	\sset^{\Omega_{cyc}^{\Sigma, op}}  \rar["f^*" description]  & \sset^{\Omega^{op}}
	\lar[bend right=20, "f_!" swap] 
	\lar[bend left=20, "f_*"] 
\end{tikzcd}
\]
when $\sset^{\Omega^{op}}$ is endowed with the model structure from \cite[Definition 6.2]{CisinskiMoerdijk:DSSIO}.
We expect that existence of this model structure should follow from the techniques of \cite{DrummondColeHackney:CFERIQMSAAIIC}, and that the equivalence could be established using \cite[Theorem 5.6]{DrummondColeHackney:CFERIQMSAAIIC}.
The chief difficulty is the first part of this: one would like to establish that $(f^*f_!, f^*f_*)$ is a Quillen adjunction, but the combinatorics of the situation are not particularly straightforward.
\end{remark}

\bibliographystyle{amsalpha}
\raggedright
\bibliography{references-2018}
\end{document}